\documentclass[11pt,reqno]{amsart}
\usepackage[usenames]{color}
\usepackage{amsmath,verbatim,graphicx,epstopdf,enumerate}
\newcommand{\D}{\mathrm{d}}

\newcommand{\lb}{\left(}

\newcommand{\rb}{\right)}
\newcommand{\PD}{\partial}

\newcommand{\Beq}{\begin{equation}}
	\newcommand{\Eeq}{\end{equation}}
\newcommand{\beq}{\begin{equation*}}
	\newcommand{\eeq}{\end{equation*}}
\newcommand{\bal}{\begin{align}}
	\newcommand{\eal}{\end{align}}

\newcommand{\bp}{\begin{prob}}
	\newcommand{\ep}{\end{prob}}
\newcommand{\bpr}{\begin{proof}}
	\newcommand{\epr}{\end{proof}}

\newcommand{\bel}[1]{\begin{equation}\label{#1}}
	\newcommand{\ee}{\end{equation}}

\newtheorem{theorem}{Theorem}[section]

\newtheorem{lemma}[theorem]{Lemma}

\theoremstyle{definition}
\newtheorem{definition}[theorem]{Definition}
\newtheorem{remark}[theorem]{Remark}

\title[Inverse problem for non-linear hyperbolic PDE]{Inverse Boundary Value Problem for a Non-linear Hyperbolic Partial Differential Equations}
\author[Nakamura and Vashisth]{Gen Nakamura$^{\dagger}$ and Manmohan Vashisth$^{\ddagger}$}
\address{$^{\dagger}$Department of Mathematics, Hokkaido University, Sapporo 060-0810, Japan.
\newline\indent E-mail:{\tt \ nakamuragenn@gmail.com}\vspace{2mm}
\newline
	{$^{\ddagger}$Beijing Computational Science Research Center, Beijing 100193, China.
	\newline
	\indent E-mail:{\tt\  mvashisth@csrc.ac.cn}}}

\begin{document}
  
  \maketitle
 	\begin{abstract}
 	In this article we are concerned with an inverse boundary value problem for a non-linear wave equation of divergence form with space dimension $n\geq 3$. In particular the so called the interior determination problem. This non-linear wave equation has
 	a trivial solution, i.e. zero solution. By linearizing this equation at the trivial solution, we have the usual linear isotropic wave equation with the speed $\sqrt{\gamma(x)}$ at each point $x$ in a given spacial domain. For any small solution $u=u(t,x)$ of this non-linear equation, we have the linear isotropic wave equation perturbed by a divergence with respect to $x$ of a vector whose components are quadratics with respect to $\nabla_x u(t,x)$ by ignoring the terms with smallness $O(|\nabla_x u(t,x)|^3)$.  We will show that we can uniquely determine $\gamma(x)$ and the coefficients of these quadratics by many boundary measurements at the boundary of the spacial domain over finite time interval. More precisely the boundary measurements are given as the so-called the hyperbolic Dirichlet to Neumann map.
 \end{abstract}
 \section{Introduction}
 \setcounter{equation}{0}
 \renewcommand{\theequation}{1.\arabic{equation}}
  Let $\Omega\subset\mathbb{R}^{n}$\,($n\geq 3$) be a bounded domain with smooth boundary $\partial\Omega$. For $T>0$, let $Q_T:=(0,T)\times\Omega$ and denote its lateral boundary by $\partial Q_T:=[0,T)\times\partial\Omega$, and also denote $[\partial Q_T]:=[0,T]\times\partial\Omega$. We will simply write $Q=Q_T,\,\partial Q=\partial Q_T$ for $T=\infty$. 
 
 Consider the following initial boundary value problem (IBVP):
 \begin{align}\label{equation of interest}
 	\begin{aligned}
 	\begin{cases}
 	&\partial_{t}^{2}u(t,x)-\nabla_{x}\cdot\vec{C}(x,\nabla_{x}u(t,x))=0, \  (t,x)\in Q_T, \\
 	&u(0,x)=\epsilon \phi_{0}(x),\,\, \partial_{t}u(0,x)=\epsilon \phi_{1}(x),  \quad  x\in\Omega,\\
 	&u(t,x)=\epsilon f(t,x), \quad (t,x)\in\partial Q_T,
 	\end{cases}
 	\end{aligned}
 	\end{align}
 where $\nabla_x:=(\partial_1,\cdots,\partial_n),\,\,\partial_j=\partial_{x_j}$ for $x=(x_1,\cdots,x_n)$. 
 Here 
 $\vec{C}(x,q)$ is given by 
 \begin{align}\label{definition of vector C}
 \begin{aligned}
 \vec{C}(x,q):= \gamma(x)q+\vec{P}(x,q)+\vec{R}(x,q)
 \end{aligned}
 \end{align}
 for vector $q:=(q_{1},\cdots,q_{n})\in\mathbb{R}^{n}$,
 $C^\infty(\overline\Omega)\ni\gamma(x)\geq C>0$ for some constant $C$, 
 \begin{align}\label{definition of P(x,q)}
 \vec{P}(x,q):=\Big(\sum_{k,l=1}^{n}c_{kl}^{1}q_{k}q_{l},\sum_{k,l=1}^{n}c_{kl}^{2}q_{k}q_{l},\sum_{k,l=1}^{n}c_{kl}^{3}q_{k}q_{l},\cdots,\sum_{k,l=1}^{n}c_{kl}^{n}q_{k}q_{l}\Big),
 \end{align}
and for simplicity we assume that each $c_{kl}^{j}\in C_{0}^\infty(\Omega)$ and $\vec{R}(x,q)\in C^\infty(\overline{\Omega}\times H)$ with $H:=\{q\in {\mathbb R}^n:\,|q|\le h\}$ for some
 constant $h>0$  satisfying the following estimate: there exists a constant $C>0$ such that
 \begin{equation}\label{estimate of R}
 |\partial_q^\alpha\partial_x^\beta\vec{R}(x,q)|\le C|q|^{3-|\alpha|}\,\,\text{for mulit-indices $\alpha,\,\beta,\,|\alpha|\le 3$}.
 \end{equation}
 
 Denote by $B^\infty(\partial Q_T)$ the completion of $C_0^\infty(\partial Q_T)$ with respect to the topology of the Fr\'echet space $C^\infty([\partial Q_T])$ with metric $d_{\partial}(\cdot,\cdot)$.
 Let $m\ge [n/2]+3$ with the largest integer  $[n/2]$ not exceeding $n/2$ and $B_M:=\{(\phi_0,\phi_1,f)\in C^\infty(\overline\Omega)^2\times B^\infty(\partial Q_T):\,d_{\partial}(0,f)+\sum_{j=1}^2 d(0,\phi_j)\le M\}$ with a metric $d(\cdot,\cdot)$ in a Fr\'echet space $C^\infty(\overline\Omega)$ and a fixed constant $M>0$, then there exists $\epsilon_0=\epsilon_0(h, T,m, M)>0$ such that
 \eqref{equation of interest} has a unique solution $u\in X_m:=\cap_{j=0}^m C^j([0,T]; H^{m-j}(\Omega))$ for any $(\phi_0,\phi_1,f)\in B_M$ and $0<\epsilon<\epsilon_0$. We refer this by the {\sl unique solvability} of \eqref{equation of interest}. Note that in particular, we can take $B_M$ as a set just consisting of $(\phi_0,\phi_1,f)\in C^\infty(\overline\Omega)^2\times C_0^\infty(\partial Q_T)$. Following \cite{Dafermos,NW} in Appendix A we will provide some argument about this together with the $\epsilon$-expansion of solutions given in Section 2. Concerning the compatibility condition for
 the initial data and boundary data, it is related to how we construct $\epsilon$-expansion. Remark \ref{two steps} gives some information about how we choose the initial data compatible to the boundary data.

 Based on this, define the Dirichlet to Neumann (DN) map $\widetilde{\Lambda}_{\vec{C}}^T$ by
 \begin{align}\label{definition of NtD map}
 	\widetilde{\Lambda}_{\vec{C}}^T(\epsilon \phi_{0},\epsilon \phi_{1},\epsilon f)=\nu(x)\cdot\vec{C}(x,\nabla_{x}u^{f})|_{\partial Q_T},\,\,(\phi_0,\phi_1,f)\in B_M,\,\,0<\epsilon<\epsilon_0
 	\end{align}
 where $u^{f}(t,x)$ is the solution to $\eqref{equation of interest}$ and $\nu(x)$ is the outer unit normal vector of $\partial\Omega$ at $x\in\partial\Omega$ directed into
 the exterior of $\Omega$. We simply denote by $\Lambda_{\overrightarrow C}^T$  the restriction $\widetilde{\Lambda}_{\vec{C}}^T\big|_{B_M^0}$, where $B_M^0:=\{(0,0,f)\in B_M :\,d_\partial(0,f)\le M\}$.
 
 The inverse problem we are going to consider is the uniqueness of identifying $\gamma=\gamma(x)$ and $\vec{P}=\vec{P}(x,q)$ from the DN map 
 $\widetilde{\Lambda}^T_{\vec{C}}$.  More precisely it is to show that if the DN maps $\widetilde{\Lambda}_{\vec{C_i}}^T,\,i=1,2$ given by $\eqref{definition of NtD map}$ for $\vec{C}=\vec{C}_i,\,i=1,2$ are the same, then $(\gamma_i,\,\vec{P_i}),\,i=1,2$ are the same, where $(\gamma_i, \vec{P_i}),\,i=1,2$ are $(\gamma, \vec{P})$ associated to $\vec{C_i},\,i=1,2$. It should be remarked here that to avoid any measurement inside $\Omega$, we will take some known $\phi_0,\,\phi_1$. We also remark that since we are assuming each $c_{kl}^{j}\in C_{0}^\infty(\Omega)$, this inverse problem can be classified as the so called interior determination. The versus of this terminology is the so called determination at the boundary. Usually for the unique identification of the coefficients of the principal part of a equation by many boundary measurements, the identification is done in two steps. The first step is the determination at the boundary and the second step is the interior determination.
 
 The non-linear wave equation of the form \eqref{equation of interest} arises as a model equation of a vibrating string with elasticity coefficient depending on strain and a model equation describing the anti-plane deformation of a uniformly thin piezoelectric material for the one spacial dimension (\cite{NWK}), and as a model equation for non-linear Love waves for the two spacial dimension (\cite{Rus}).
 
 There are several works on inverse problems for non-linear wave equations. For example, 
 Denisov \cite{Denisov quasilinear}, Grasselli \cite{Grasselli} and Lorenzi-Paparoni \cite{Lorenzi Paparoni} considered the inverse problems related to non-linear wave equations, but non-linearity in their works is in lower order terms. Under the same set up as our inverse problem except the space dimension, Nakamura-Watanabe in \cite{NW} identified $(\gamma, \vec{P})$ by giving a reconstruction formula in one space dimension which also gives uniqueness. We are going to prove the uniqueness for our inverse problem when the space dimension $n\geq 3$.
 The physical meaning of our inverse problem can be considered as a problem to identify especially the higher order tensors in non-linear elasticity for its simplified model equation. In a smaller scale the higher order tensors become important.There is a recent uniqueness result \cite{hoop} on the Laudau-Lifshitz nonlinear eastic model.
 
 We will also mention about some related works for elliptic and parabolic equations. For elliptic equations, Kang-Nakamura in \cite{KN} studied the uniqueness for determining the non-linearity in conductivity equation. Our result can be viewed as a generalisation of \cite{KN} for non-linear wave equation. There are other works related to non-linear elliptic PDE, we refer to \cite{Hervas and Sun, Isakov Uniqueness,Isakov and Nachman,Nakamura Sun,Sun Uhlmann,Sun Conjecture,Sun Semilinear}. For parabolic equations, we refer to \cite{COY,Isakov Uniqueness,Klibanov}.
 
 \medskip
 In order to state our main result, we need to define the filling time $T^*$. 
 \begin{definition}\label{filling time}
 	Let $E^T$ be the maximal subdomain of $\Omega$ such that any solution $v=v(t,x)$ of $\partial_t^2 v-\nabla_x\cdot(\gamma\nabla_x v)=0$ in $Q_T$ will become zero in this subdomain at $t=2T$ if the Cauchy data of $v$ on $\partial Q_{2T}$ are zero. We call $E^T$ the influence domain. Then define the filling time $T^*$ by $T^*=\text{\rm inf}\{T>0:\,E^T=\Omega\}$. 
 \end{definition}
 
 By the Holmgren-John-Tataru unique continuation property of solutions of the above equation for $v$ in Definition \ref{filling time}, there exists a finite filling time $T^*$ (see \cite{KKL} and the references there in for further details).
 Based on this we have the following main theorem.
 \begin{theorem}\label{Uniqueness theorem}
 	For $i=1,2$, let 
 	\begin{align*}
 	\vec{P}^{(i)}(x,q):=\Big(\sum_{k,l=1}^{n}c_{kl}^{1(i)}q_{k}q_{l},\sum_{k,l=1}^{n}c_{kl}^{2(i)}q_{k}q_{l},\sum_{k,l=1}^{n}c_{kl}^{3(i)}q_{k}q_{l},\cdots,\sum_{k,l=1}^{n}c_{kl}^{n(i)}q_{k}q_{l}\Big)
 	\end{align*}
 	and  $\vec{C}^{(i)}(x,q)=\gamma_{i}(x)p+\vec{P}^{(i)}(x,q)+\vec{R}^{(i)}(x,q)$ with $\gamma_i,\,\vec{P}^{(i)}\ \text{and} \ \vec{R}^{(i)},\,i=1,2$ satisfying the same conditions as for $\gamma,\,\vec{P}\ \text{and}\ \vec{R}$. Further let $u^{(i)},\,i=1,2$ be the solutions to the following IBVP:
 	\begin{align}\label{equation for ui}
 	\begin{aligned}
 	\begin{cases}
 	&\partial_{t}^{2}u^{(i)}(t,x)-\nabla_{x}\cdot\vec{C}^{(i)}(x,\nabla_{x}u^{(i)}(t,x))=0, \   (t,x)\in Q_T, \\
 	&u^{(i)}(0,x)=\epsilon\phi_{0}(x),\ \partial_{t}u^{(i)}(0,x)=\epsilon \phi_{1}(x),\ \  x\in\Omega,\\
 	&u^{(i)}(t,x)=\epsilon f(t,x), \quad (t,x)\in\partial Q_T
 	\end{cases}
 	\end{aligned}
 	\end{align}
 	with any $0<\epsilon<\epsilon_0$.
 	Assume $T>2T^*$ and let $\widetilde{\Lambda}_{\vec{C}^{(1)}}^T$ and $\widetilde{\Lambda}_{\vec{C}^{(2)}}^T$  be the DN maps as defined in $\eqref{definition of NtD map}$  corresponding to $u^{(1)}$ and $u^{(2)}$ respectively. Assume that 
 	\begin{align}\label{equality of NtD map}
 	\widetilde{\Lambda}_{\vec{C}^{(1)}}^T(\epsilon\phi_0,\epsilon\phi_1,\epsilon f)=\widetilde{\Lambda}_{\vec{C}^{(2)}}^T(\epsilon\phi_0,\epsilon\phi_1, \epsilon f),\,\,(\phi_0,\phi_1,f)\in B_M,\,\,0<\epsilon<\epsilon_0.
 	\end{align}
 	Then we have
 	\begin{align*}
 	\gamma_{1}(x)=\gamma_{2}(x),\,\, c_{kl}^{j(1)}(x)=c_{kl}^{j(2)}(x),\,\,x\in\Omega,\,\,1\leq j,k,l\leq n.
 	\end{align*}
 	
 \end{theorem}
 
 \begin{remark}\label{two steps}${}$
 The proof will be done in two steps. Namely
 we first show that from $\Lambda_{\vec{C}^{(1)}}^T(\epsilon f)=\Lambda_{\vec{C}^{(2)}}^T(\epsilon f)$, $f\in B_M^0$, we can have $\gamma=\gamma_1=\gamma_2$ and $\gamma$ can be reconstructed from the linearization $\Lambda_\gamma^T$ of these two given DN maps which end up with this same one. Next from \eqref{equality of NtD map} with some initial data which can be generated by $\gamma$ and the boundary data, we prove the unique identification of $c_{kl}^j$'s. 
 \end{remark}
 The most difficult part of proving Theorem \ref{Uniqueness theorem} is showing the uniqueness of identifying the quadratic nonlinear part $\vec{P}(x,q)$. The key ingredients for showing this are to use the control with delay in time (see \eqref{delay control}) and the special polarization for the difference of the quadratic nonlinear part with integration with respect to the delay time (see \eqref{equation for u2 after even extension}) coming from two $\vec{C}^{(i)}(x, q),\,i=1,2$ so that via the Laplace tranform with respect to $t$, we can relate the problem of identifying the quadratic part to that for a nonlinear elliptic equation. The reduced problem is almost the same as the one considered in \cite{KN}. 
 
 \medskip
 The rest of this paper is organized as follows. In Section 2, we will introduce the $\epsilon$-expansion of the IBVP to analyze the hyperbolic DN map. As a consequence, we will show that the hyperbolic DN map determines the hyperbolic DN map associated with the equation $\partial_t^2 v-\nabla_x\cdot(\gamma\nabla_x v)=0$ in $(0,T)\times\Omega$. This immediately implies the uniqueness of identifying $\gamma$. Section 3 is devoted to proving the uniqueness of identifying the quadratic non-linear part of $\vec{P}(x,q)$. In Appendix A---D, we will give some arguments for the unique solvability of \eqref{equation of interest}, the justification of $\epsilon$-expansion, the necessary tools and some estimate.
 
 \section{$\epsilon$-expansion of the solution to the IBVP}
 \setcounter{equation}{0}
 \renewcommand{\theequation}{2.\arabic{equation}}
 To prove the theorem, we will use  the $\epsilon$-expansion of the solution $u^{(i)f}$ to Equation  $\eqref{equation for ui}$ which is given by 
 \begin{align}\label{epsilon approximation of solution of u i}
 u^{(i)f}(t,x)=\epsilon u^{(i)f}_{1}(t,x)+\epsilon^{2}u^{(i)f}_{2}(t,x)+O(\epsilon^3).     \end{align}
 Here $O(\epsilon^{3})$ we mean the following:
 \begin{align*}
     \begin{aligned}
      w(t, x)=O(\epsilon^{3}) \mbox{ means that} \  \sup_{0\leq t\leq T}\sum_{k=0}^{m}\lVert w^{(k)}(t,.)\rVert_{m-k}^{2}=O(\epsilon^{3}),\\
\mbox{where} \ \lVert.\rVert_{k} \  \mbox{is the norm of the usual Sobolev space}\,  H^{k}(\Omega).
    \end{aligned}
 \end{align*}
 We will provide some argument on the justification of this expansion in Appendix A. 
 
 By the straight forward calculation, we have the followings:
 \begin{align*}
 \begin{aligned}
 \partial_{t}^{2}u^{(i)f}=\epsilon\partial_{t}^{2}u^{(i)f}_{1}(t,x)+\epsilon^{2}\partial_{t}^{2}u^{(i)f}_{2}(t,x)+O(\epsilon^3),\qquad\qquad\qquad\qquad\qquad\qquad\qquad\qquad\\
 \nabla_{x}u^{(i)f}=\epsilon \nabla_{x}u^{(i)f}_{1}(t,x)+\epsilon^{2}\nabla_{x}u^{(i)f}_{2}(t,x)+O(\epsilon^3),\qquad\qquad\qquad\qquad\qquad\qquad\qquad\qquad\\
 \vec{C}^{(i)}(x,\nabla_{x}u^{(i)f})=\gamma_{i}(x)\nabla_{x}u^{(i)f}(t,x)+\qquad\qquad\qquad\qquad\qquad\qquad\qquad\qquad\qquad\qquad\qquad\qquad\\
 \Big(\sum_{k,l=1}^{n}c_{kl}^{1(i)}\partial_{k}u^{(i)f}\partial_{l}u^{(i)f},\sum_{k,l=1}^{n}c_{kl}^{2(i)}\partial_{k}u^{(i)f}\partial_{l}u^{(i)f},\cdots,\sum_{k,l=1}^{n}c_{kl}^{n(i)}\partial_{k}u^{(i)f}\partial_{l}u^{(i)f}\Big)+o(\epsilon^{3})\qquad\quad\\
 =\epsilon\gamma_{i}(x)\nabla_{x}u^{(i)f}_{1}+\epsilon^{2}\gamma_{i}(x)\nabla_{x}u^{(i)f}_{2}+\epsilon^{2}\Big(\sum_{k,l=1}^{n}c_{kl}^{j(i)}\partial_{k}u^{(i)f}_{1}\partial_{l}u^{(i)f}_{1}\Big)_{1\leq j\leq n}+O(\epsilon^{3}),\qquad\qquad\\
 \nabla_{x}\cdot\vec{C}^{(i)}(x,\nabla_{x}u^{(i)f})= \epsilon\nabla_{x}\cdot(\gamma_{i}(x)\nabla_{x}u^{(i)f}_{1})+\epsilon^{2}\nabla_{x}\cdot(\gamma_{i}(x)\nabla_{x}u^{(i)f}_{2})\qquad\qquad\qquad\qquad\qquad\\
 +\epsilon^{2}\sum_{j=1}^{n}\partial_{j}\Big(\sum_{k,l=1}^{n}c_{kl}^{j(i)}\partial_{k}u^{(i)f}_{1}\partial_{l}u^{(i)f}_{1}\Big)+O(\epsilon^{3}).\qquad\qquad\qquad
 \end{aligned}
 \end{align*}
 Substitute $\eqref{epsilon approximation of solution of u i}$ into $\eqref{equation for ui}$, and arrange the terms into ascending order of power of $\epsilon$ by using the above calculations. Then setting the coefficients of $\epsilon$ and $\epsilon^2$ equal zero, we have the following equations for $u_{1}^{(i)}=u_{1}^{(i)f}$ and $u_{2}^{(i)}=u_{2}^{(i)f}$: 
 \begin{align}\label{equation for u1 i}
 	\begin{aligned}
 	\begin{cases}
 	&\partial_{t}^{2}u^{(i)}_{1}(t,x)-\nabla_{x}\cdot(\gamma_{i}(x)\nabla_{x}u^{(i)}_{1}(t,x))=0, \ \   (t,x)\in Q_{T}\\
 	&u^{(i)}_{1}(0,x)=0,\ \partial_{t}u^{(i)}_{1}(0,x)=0, \ \  x\in\Omega\\
 	&u_{1}^{(i)}(t,x)=f(t,x),\ \ (t,x)\in\partial Q_{T},
 	\end{cases}
 	\end{aligned}
 	\end{align}
 	\begin{align}\label{equation for u2 i}
 	\begin{aligned}
 	\begin{cases}
 	&\partial_{t}^{2}u^{(i)}_{2}(t,x)-\nabla_{x}\cdot(\gamma_{i}(x)\nabla_{x}u^{(i)}_{2}(t,x))=\nabla_{x}\cdot \vec{P}^{(i)}(x,\nabla_{x}u^{(i)f}_{1}), \  (t,x)\in Q_{T}\\
 	&u^{(i)}_{2}(0,x)=\partial_{t}u^{(i)}_{2}(0,x)=0,\ \   x\in\Omega\\
 	&u_{2}^{(i)}(t,x)=0,\quad   \ (t,x)\in \partial Q_{T}.
 	\end{cases}
 	\end{aligned}
 	\end{align}
 For the well-posedness of these initial boundary value problem see for example Theorem 2.45 of \cite{KKL}.

 By using the $\epsilon$-expansion $\eqref{epsilon approximation of solution of u i}$ of solution to equation $\eqref{equation for ui}$, we have the $\epsilon$-expansion of the DN map: 
 \begin{align}\label{NtD map in epsilon aexpansion for u i}
 	\begin{aligned}
 	\Lambda_{\vec{C}^{(i)}}^T(\epsilon f)&=\epsilon\left(\gamma_{i}(x)\partial_{\nu}u^{(i)}_{1}(t,x)\right)\Big|_{\partial Q_{T}}\\
 	&\ \ \ \ \ +\epsilon^{2}\left(\gamma_{i}(x)\partial_{\nu}u^{(i)}_{2}(t,x)+\nu(x)\cdot\vec{P}^{(i)}(x,\nabla_{x}u_{1}^{(i)f})\right)\Big|_{\partial Q_{T}}+O(\epsilon^{3})\\
 	&=\epsilon g_{1}^{(i)}+\epsilon^{2}g_{2}^{(i)}+O(\epsilon^{3}).
 	\end{aligned}
 	\end{align}
 This gives us 
 \begin{align}\label{NtD map for u1 i}
 	\Lambda^{T}_{\gamma_{i}}(f)= \gamma_{i}(x)\partial_{\nu}u^{(i)f}_{1}|_{\partial Q_{T}}=g^{(i)}_{1}(t,x)|_{\partial Q_{T}},\,\, (t,x)\in\partial Q_{T}
 	\end{align}
 where each $\Lambda_{\gamma_{i}}^T$ is the DN map associated to the initial boundary value problem \eqref{equation for u1 i} defined by 
 \begin{equation}\label{linear ND map}
 	\Lambda^{T}_{\gamma_{i}}(f)=\gamma_{i}(x)\partial_{\nu}u^{(i)f}_{1}\big|_{\partial Q_T},\,\,f\in C_0^\infty(\partial Q_T).
 	\end{equation}
 
 Therefore we have shown the following implication:
 \begin{equation}\label{implication}
 \begin{aligned}
& \Lambda_{\vec{C}^{(1)}}^T(\epsilon f)=\Lambda_{\vec{C}^{(2)}}^T(\epsilon f),\,\,f\in B_M^0,\,\,0<\epsilon<\epsilon_0\\
& \Longrightarrow \Lambda_{\gamma_{1}}^T=\Lambda_{\gamma_{2}}^T.
 \end{aligned}
 \end{equation}
 
 \section{Proof for Theorem $\ref{Uniqueness theorem}$}
 \subsection{Proof of the uniqueness for $\gamma$}
 \setcounter{equation}{0}
 \renewcommand{\theequation}{3.\arabic{equation}}${}$
 \par
 By knowing each $\Lambda_{\vec{C}^{(i)}}^T(\epsilon f),\,i=1,2$ for any $f\in B_M^0$, $0<\epsilon<\epsilon_0$, we do know each $\Lambda_{\gamma_{i}}^T,\,i=1,2$ from \eqref{NtD map for u1 i}. Then, recalling $T>2T^*$, we can reconstruct each $\gamma_{i},\, i=1,2$ from $\Lambda_{\gamma_{i}}^T,\,i=1,2$ by the boundary control method (see \cite{Belishev and Kurylev}). By \eqref{implication}
 the reconstructed $\gamma_{i},\,i=1,2$ in $\Omega$ are the same. We denote this common $\gamma_{i},\,i=1,2$ by $\gamma$, i.e.
 \begin{equation}\label{gamma}
 \gamma=\gamma_1=\gamma_2\,\,\,\text{in\, $\Omega$}.
 \end{equation}
 Together with this and the given Dirichlet data $f$ is the same for $u_1^{(i)},\,i=1,2$, we do know
 $$
 u_1^{(1)}=u_1^{(2)}\,\,\text{in $Q_T$}.
 $$
 Due to the fact that $\gamma$ is independent of $t$, this implies
 $$
 u_1^{(1)}=u_1^{(2)}\,\,\text{in $Q$}.
 $$
 We denote this common solution by $u_1=u_1^f$, i.e.
 \begin{equation}\label{u_1}
 u_1=u_1^f= u_1^{(1)}=u_1^{(2)}\,\,\text{in $Q$}.
 \end{equation}

 \subsection{Proof of the uniqueness for $c_{kl}^{j}(x)$}
 \setcounter{equation}{0}
 \renewcommand{\theequation}{3.\arabic{equation}}${}$
 \par
 Through out this subsection we assume that the Dirichlet data $f(t,x),\,g(t,x)$ are of the following forms:
 \begin{equation}\label{form of f,g}
     f(t,x)=\chi(t)\widetilde f(x),\,\,g(t,x)=\chi(t)\widetilde g(x),\,\,\widetilde f,\,\widetilde g\in C^\infty(\partial\Omega),
 \end{equation}
 where $\chi\in C_0^\infty([0,T))$ is such that its Laplace transform $\widehat\chi(\tau)=\int_0^\infty e^{-\tau t}\chi(t)\,dt$ has the asymptotic 
 \begin{equation}\label{asymptotic of chi}
  \widehat\chi(\tau)=\tau^{-\mu} (1+O(\tau^{-1})),\,\,\tau\rightarrow\infty  
 \end{equation} with $\mu\in\mathbb{N}$ for $\tau_R:=\text{Re}\,\tau\ge1$. We abuse the notations to denote $c_{kl}^{j}(x):=c_{kl}^{j(1)}(x)-c_{kl}^{j(2)}(x)$ so that $\vec{P}(x,q):=\vec{P}^{(1)}(x,q)-\vec{P}^{(2)}(x,q)=\Big(\sum_{k,l=1}^{n}c^{j}_{kl}q_{k}q_{l}\Big)_{1\leq j\leq n}$ and $u^{f}_{2}(t,x)=u_{2}^{(1)f}(t,x)-u_{2}^{(2)f}(t,x)$. Then, from $\eqref{equation for u1 i}$ and $\eqref{equation for u2 i}$, $u_1=u^{f}_{1}(t,x)$ and $u_2=u^{f}_{2}(t,x)$ are the solutions to  the following initial boundary value problems:
 \begin{align}\label{equation for u 1 in uniqueness case}
 	\begin{aligned}
 	\begin{cases}
 	&\partial_{t}^{2}u_{1}(t,x)-\nabla_{x}\cdot(\gamma(x)\nabla_{x}u_{1}(t,x))=0, \   (t,x)\in Q,\\
 	&u_{1}(0,x)=0,\ \partial_{t}u_{1}(0,x)=0, \ x\in\Omega,\\
 	&u_{1}(t,x)=f(t,x),\  (t,x)\in\partial Q,
 	\end{cases}
 	\end{aligned}
 	\end{align}
 and
 \begin{align}\label{equation for u 2 uniqueness case}
 \begin{aligned}
 \begin{cases}
 &\partial_{t}^{2}u_{2}(t,x)-\nabla_{x}\cdot(\gamma(x)\nabla_{x}u_{2}(t,x))=\nabla_{x}\cdot(\vec{P}(x,\nabla_{x}u_{1}(t,x))), \ (t,x)\in Q,\\
 &u_{2}(0,x)=\partial_{t}u_{2}(0,x)=0,\   x\in\Omega,\\
 &u_{2}(t,x)=0,\ (t,x)\in \partial Q,
 \end{cases}
 \end{aligned}
 \end{align}
 respectively.
 We emphasize here that $u_1,\, u_2\in C^\infty ([0,\infty)\times\overline{\Omega})$ are the unique solutions to \eqref{equation for u 1 in uniqueness case} and \eqref{equation for u 2 uniqueness case}, respectively.\\
 
 From the equality of DN map
 in \eqref{NtD map in epsilon aexpansion for u i}, we have 
 \begin{align}\label{equality of Neumann data at epsilon2 level}
 [\gamma(x)\partial_{\nu}u^{f}_{2}(t,x)+\nu(x)\cdot\vec{P}(x,\nabla_{x}u^{f}_{1}(t,x))]\Big|_{\partial Q}=0.
 \end{align}
 Consider the even extension of $f(t,\cdot)$ with respect to $t$, so that the extended $f(t,\cdot)$ is defined on $\mathbb{R}\times\overline{\Omega}$. By abusing the notation, we denote this extended $f(t,\cdot)$ by the same notation. We also define $Y_{s}$ for any fixed $s\in\mathbb{R}$ by
 \begin{align}\label{delay control}
 Y_{s}f(t,.):=f(t-s,\cdot),\ \ t \in\mathbb{R}.
 \end{align}
 This is a control with delay time $s$.
 Then we have $u^{Y_{s}f}(t,x)=u^{f}(t-s,x)$. 
 
 For $f,\, g\in B_M^0$ consider the solutions $u_1^f$ and $u_1^g$ of \eqref{equation for ui} with $\gamma_i=\gamma$ and Dirichlet data $f$ and $g\in B_M^0$, respectively. Then by extending $f,\, g$ to even functions over $\mathbb{R}$, with respect to $t$, we can have the solutions $u_1^f$ and $u_1^g$ of \eqref{equation for u 1 in uniqueness case} with respective Dirichlet data $f$ and $g$ which are the even extension of the original $f, g$. Thus we can have $u_1^{f\pm Y_sg}(t,s)=u_1^f(t,x)\pm u_1^{Y_s g}(t,x)$. Now for the initial data $\epsilon\phi_0:=\epsilon (u_1^{f+Y_sg}-u_1^{f-Y_sg})\big|_{t=0}=2\epsilon Y_{s}g\big|_{t=0}\,,\,\,\epsilon\phi_1:=\{\partial_t(u_1^{f+Y_sg}-u_1^{f-Y_sg})\}\big|_{t=0}=2\epsilon(\partial_t u_1^{Y_s g})\big|_{t=0}$ and the Dirichlet data $(u_1^{f+Y_sg}-u_1^{f-Y_sg})\big|_{\partial\Omega}=2\epsilon Y_s g$, we can consider the initial boundary value problem \eqref{equation for ui} with $\gamma_i=\gamma$ and we can also consider its $\epsilon$-expansion.
Then, likewise $u_2^f$, it is easy to see that $u_2(t,x;s)=u_2^{f+Y_sg}(t,x)-u_2^{f-Y_sg}(t,x)$ satisfies the following initial boundary value problem:
 \begin{align}\label{equation for u2 after even extension}
 	\begin{aligned}
 	\begin{cases}
 	&\partial_{t}^{2}u_{2}(t,x;s)-\nabla_{x}\cdot(\gamma(x)\nabla_{x}u_{2}(t,x;s))=\\
 	&\qquad 2\sum_{j=1}^{n}\partial_{j}\Big(\sum_{k,l=1}^{n}c_{kl}^{j}(x)\partial_{k}u_{1}^{f}(t,x)\partial_{l}u_{1}^{g}(s-t,x)\Big)\\
 	&\qquad +2\sum_{j=1}^{n}\partial_{j}\Big(\sum_{k,l=1}^{n}c_{kl}^{j}(x)\partial_{k}u_{1}^{g}(s-t,x)\partial_{l}u_{1}^{f}(t,x)\Big),\ \ t,s\in\mathbb{R},\ x\in\Omega, \\
 	&u_{2}(0,x;s)=\partial_{t}u_{2}(0,x;s)=0,\ x\in\Omega,\ s\in\mathbb{R},\\
 	&u_{2}(t,x;s)=0,\ t,s\in\mathbb{R},\, x\in\partial\Omega.
 	\end{cases}
 	\end{aligned}
 	\end{align}
  Also from \eqref{equality of NtD map}, we have 
 \begin{align}\label{Neumann data of u_2}
 	\begin{aligned}
 	&\gamma(x)\partial_{\nu}u_{2}(t,x;s)+2\sum_{j=1}^{n}\nu_{j}(x)\Big(\sum_{k,l=1}^{n}c_{kl}^{j}(x)\partial_{k}u_{1}^{f}(t,x)\partial_{l}u_{1}^{g}(s-t,x)\Big)\\
 	&+2\sum_{j=1}^{n}\nu_{j}(x)\Big(\sum_{k,l=1}^{n}c_{kl}^{j}(x)\partial_{k}u_{1}^{g}(s-t,x)\partial_{l}u_{1}^{f}(t,x)\Big)=0.
 	\end{aligned}
 	\end{align}
 Now let us denote by
 \begin{align*}
 \begin{aligned}
 &Bv(t,x):= -\nabla_{x}\cdot(\gamma(x)\nabla_{x}v(t,x))\\
 &F(t,x;s):= 2\sum_{j=1}^{n}\partial_{j}\Big(\sum_{k,l=1}^{n}c_{kl}^{j}(x)\partial_{k}u_{1}^{f}(t,x)\partial_{l}u_{1}^{g}(s-t,x)\Big)\\
 &\ \ \ \ \ \ \ \qquad \qquad +2\sum_{j=1}^{n}\partial_{j}\Big(\sum_{k,l=1}^{n}c_{kl}^{j}(x)\partial_{k}u_{1}^{g}(s-t,x)\partial_{l}u_{1}^{f}(t,x)\Big).
 \end{aligned}
 \end{align*}
 We consider $B$ as a selfadjoint positive unbounded operator on $L^2(\Omega)$ with domain $D(B):=\{v\in L^2(\Omega): Bv\in L^2(\Omega),\,v\big|_{\partial\Omega}=0\}$. 
 Now using representation formula for the solution to \eqref{equation for u2 after even extension}, we have 
 \begin{align}\label{Representation formula for u2}
 \begin{aligned}
 &u_{2}(t,x;s)=B^{-1/2}\int\limits_{0}^{t}\sin\left\{(t-\sigma)B^{1/2}\right\}F(\sigma,x;s)d\sigma,\\
 &\partial_{s}u_{2}(t,x;s)=B^{-1/2}\int\limits_{0}^{t}\sin\left\{(t-\sigma)B^{1/2}\right\}\partial_{s}F(\sigma,x;s)d\sigma,\\
 &\partial_{s}^{2}u_{2}(t,x;s)=B^{-1/2}\int\limits_{0}^{t}\sin\left\{(t-\sigma)B^{1/2}\right\}\partial^{2}_{s}F(\sigma,x;s)d\sigma,\\
 &\partial_{t}u_{2}(t,x;s)=\int\limits_{0}^{t}\cos\left\{(t-\sigma)B^{1/2}\right\}F(\sigma,x;s)d\sigma,
 \end{aligned}
 \end{align}
 where $B^{\pm 1/2}$, $\sin\{(t-\sigma)B^{1/2}\}$ and $\cos\{(t-\sigma)B^{1/2}\}$ are defined using the spectral decomposition of $B$.
 
 Next let  
 \begin{align*}
 	u_2^{(-1)}(s,x)=\int\limits_{0}^{s}u_2(t,x;s)dt.
 	\end{align*}
 We want to derive an equation for $u_2^{(-1)}(s,x)$. To begin with, we have
 \begin{align*}
 \begin{aligned}
 &\partial_{s}u_{2}^{(-1)}(s,x)=u_{2}(s,x;s)+\int\limits_{0}^{s}\partial_{s}u_{2}(t,x;s) dt\\
 &\partial_{s}^{2}u_{2}^{(-1)}(s,x)=\partial_{t}u_{2}(s,x;s)+2\partial_{s}u_{2}(s,x;s)+\int\limits_{0}^{s}\partial_{s}^{2}u_{2}(t,x;s)dt\\
 &\nabla_{x}\cdot(\gamma(x)\nabla_{x}u_{2}^{(-1)}(s,x))=\int\limits_{0}^{s}\nabla_{x}\cdot(\gamma(x)\nabla_{x}u_{2}(t,x;s))dt.
 \end{aligned}
 \end{align*}
 Now using \eqref{Representation formula for u2}, we have 
 \begin{align*}
 \begin{aligned}
 &J_{1}(s,x):=\partial_{t}u_{2}(s,x;s)=\int\limits_{0}^{s}\cos\{(s-\sigma)B^{1/2}\}F(\sigma,x;s)d\sigma\\
 &\ = 2\sum_{j=1}^{n}\int\limits_{0}^{s}\cos\left\{(s-\sigma)B^{1/2}\right\}\partial_{j}\Bigg[\sum_{k,l=1}^{n}c_{kl}^{j}(x)\left\{\partial_{k}u_{1}^{f}(\sigma,x)\partial_{l}u_{1}^{g}(s-\sigma,x)+\partial_{k}u_{1}^{g}(s-\sigma,x)\partial_{l}u_{1}^{f}(\sigma,x)
 \right\}\Bigg]d\sigma.
 \end{aligned}
 \end{align*}
Define by  $J_{2}(s,x)$ as 
  \begin{align*}
     \begin{aligned}
 J_{2}(s,x)&:=2\partial_{s}u_{2}(s,x;s)=2B^{-1/2}\int\limits_{0}^{s}\sin\left\{(s-\sigma)B^{1/2}\right\}\partial_{s}F(\sigma,x;s)d\sigma\\
 &\ \ =2B^{-1/2}\int\limits_{0}^{s}\sin\left\{(s-\sigma)B^{1/2}\right\}\Bigg[2\sum_{j=1}^{n}\partial_{j}\Big(\sum_{k,l=1}^{n}c_{kl}^{j}(x)\partial_{k}u_{1}^{f}(\sigma,x)(\partial_{l}u_{1}^{g})'(s-\sigma,x)\Big)\\
 &\ \ \ \ \ \ \ \qquad \qquad +2\sum_{j=1}^{n}\partial_{j}\Big(\sum_{k,l=1}^{n}c_{kl}^{j}(x)(\partial_{k}u_{1}^{g})'(s-\sigma,x)\partial_{l}u_{1}^{f}(\sigma,x)\Big)\Bigg]d\sigma
 \end{aligned}
 \end{align*}
 where ${}'$ denote the single derivative with respect to $s$. We will denote by  ${}''$  double derivatives with respect to $s$. Now using the integration by parts, $J_{2}(s,x)$ becomes
 \begin{align*}
 \begin{aligned}
 J_{2}(s,x)&=2\int\limits_{0}^{s}\cos\left\{(s-\sigma)B^{1/2}\right\}\Bigg[2\sum_{j=1}^{n}\partial_{j}\Big(\sum_{k,l=1}^{n}c_{kl}^{j}(x)\partial_{k}u_{1}^{f}(\sigma,x)\partial_{l}u_{1}^{g}(s-\sigma,x)\Big)\\
 &\ \ \ \ \ \ \ \qquad \qquad +2\sum_{j=1}^{n}\partial_{j}\Big(\sum_{k,l=1}^{n}c_{kl}^{j}(x)\partial_{k}u_{1}^{g}(s-\sigma,x)\partial_{l}u_{1}^{f}(\sigma,x)\Big)\Bigg]d\sigma\\
 &\ \ -2B^{-1/2}\int\limits_{0}^{s}\sin\left\{(s-\sigma)B^{1/2}\right\}\Bigg[2\sum_{j=1}^{n}\partial_{j}\Big(\sum_{k,l=1}^{n}c_{kl}^{j}(x)(\partial_{k}u_{1}^{f})'(\sigma,x)\partial_{l}u_{1}^{g}(s-\sigma,x)\Big)\\
 &\ \ \ \ \ \ \ \qquad \qquad +2\sum_{j=1}^{n}\partial_{j}\Big(\sum_{k,l=1}^{n}c_{kl}^{j}(x)\partial_{k}u_{1}^{g}(s-\sigma,x)(\partial_{l}u_{1}^{f})'(\sigma,x)\Big)\Bigg]d\sigma.\\
 \end{aligned}
 \end{align*}
Also 
 \begin{align*}
     \begin{aligned}
 &J_{3}(s,x):=\int\limits_{0}^{s}\partial_{s}^{2}u_{2}(t,x;s)dt=\int\limits_{0}^{s}B^{-1/2}\int\limits_{0}^{t}\sin\left\{(t-\sigma)B^{1/2}\right\}\partial_{s}^{2}F(\sigma,x;s)d\sigma\\
 &\ =\int\limits_{0}^{s}\Bigg(B^{-1/2}\int\limits_{0}^{t}\sin\left\{(t-\sigma)B^{1/2}\right\}\Bigg[2\sum_{j=1}^{n}\partial_{j}\Big(\sum_{k,l=1}^{n}c_{kl}^{j}(x)\partial_{k}u_{1}^{f}(\sigma,x)(\partial_{l}u_{1}^{g})''(s-\sigma,x)\Big)\\
 &\ \ \ \ \ \ \ \qquad \qquad +2\sum_{j=1}^{n}\partial_{j}\Big(\sum_{k,l=1}^{n}c_{kl}^{j}(x)(\partial_{k}u_{1}^{g})''(s-\sigma,x)\partial_{l}u_{1}^{f}(\sigma,x)\Big)\Bigg]d\sigma\Bigg) dt,\\ 
 \end{aligned}
 \end{align*}

 and After using integration by parts $J_{3}(s,x)$ becomes 
 \begin{align*}
 \begin{aligned}
 J_{3}(s,x)&=\int\limits_{0}^{s}\Bigg(\int\limits_{0}^{t}\cos\left\{(t-\sigma)B^{1/2}\right\}\Bigg[2\sum_{j=1}^{n}\partial_{j}\Big(\sum_{k,l=1}^{n}c_{kl}^{j}(x)\partial_{k}u_{1}^{f}(\sigma,x)(\partial_{l}u_{1}^{g})'(s-\sigma,x)+\Big)\\
 &\ \  +2\sum_{j=1}^{n}\partial_{j}\Big(\sum_{k,l=1}^{n}c_{kl}^{j}(x)(\partial_{k}u_{1}^{g})'(s-\sigma,x)\partial_{l}u_{1}^{f}(\sigma,x)\Big)\Bigg]d\sigma\Bigg) dt\\
 &\ \ -\int\limits_{0}^{s}\Bigg(B^{-1/2}\int\limits_{0}^{t}\sin\left\{(t-\sigma)B^{1/2}\right\}\Bigg[2\sum_{j=1}^{n}\partial_{j}\Big(\sum_{k,l=1}^{n}c_{kl}^{j}(x)(\partial_{k}u_{1}^{f})'(\sigma,x)(\partial_{l}u_{1}^{g})'(s-\sigma,x)\Big)\\
 &\ \ \ \ \ \ \ \qquad \qquad +2\sum_{j=1}^{n}\partial_{j}\Big(\sum_{k,l=1}^{n}c_{kl}^{j}(x)(\partial_{k}u_{1}^{g})'(s-\sigma,x)(\partial_{l}u_{1}^{f})'(\sigma,x)\Big)\Bigg]d\sigma\Bigg) dt.\\
 \end{aligned}
 \end{align*}
 Using again integration by parts, we have 
 \begin{align*}
     \begin{aligned}
 J_{3}(s,x)&=\int\limits_{0}^{s}2\sum_{j=1}^{n}\partial_{j}\Bigg[\sum_{k,l=1}^{n}c_{kl}^{j}(x)\partial_{k}u_{1}^{f}(t,x)\partial_{l}u_{1}^{g}(s-t,x)+\sum_{k,l=1}^{n}c_{kl}^{j}(x)\partial_{k}u_{1}^{g}(s-t,x)\partial_{l}u_{1}^{f}(t,x)\Bigg] dt\\
 &\ \ -\int\limits_{0}^{s}\Bigg(B^{1/2}\int\limits_{0}^{t}\sin\left\{(t-\sigma)B^{1/2}\right\}\Bigg[2\sum_{j=1}^{n}\partial_{j}\Big(\sum_{k,l=1}^{n}c_{kl}^{j}(x)\partial_{k}u_{1}^{f}(\sigma,x)\partial_{l}u_{1}^{g}(s-\sigma,x)\Big)\\
 &\ \ \ \ \ \ \ \  \ \ \ \ \ \ \ \ +2\sum_{j=1}^{n}\partial_{j}\Big(\sum_{k,l=1}^{n}c_{kl}^{j}(x)\partial_{k}u_{1}^{g}(s-\sigma,x)\partial_{l}u_{1}^{f}(\sigma,x)\Big)\Bigg]d\sigma\Bigg) dt\\
 &\ \ -\int\limits_{0}^{s}\Bigg(\int\limits_{0}^{t}\cos\left\{(t-\sigma)B^{1/2}\right\}\Bigg[2\sum_{j=1}^{n}\partial_{j}\Big(\sum_{k,l=1}^{n}c_{kl}^{j}(x)(\partial_{k}u_{1}^{f})'(\sigma,x)\partial_{l}u_{1}^{g}(s-\sigma,x)\Big)\\
 &\ \ \ \ \ \ \ \ \ \ \ \ \ \ \ \   \ +2\sum_{j=1}^{n}\partial_{j}\Big(\sum_{k,l=1}^{n}c_{kl}^{j}(x)\partial_{k}u_{1}^{g}(s-\sigma,x)(\partial_{l}u_{1}^{f})'(\sigma,x)\Big)\Bigg]d\sigma\Bigg) dt\\
 &\ \ -\int\limits_{0}^{s}\Bigg(B^{-1/2}\int\limits_{0}^{t}\sin\left\{(t-\sigma)B^{1/2}\right\}\Bigg[2\sum_{j=1}^{n}\partial_{j}\Big(\sum_{k,l=1}^{n}c_{kl}^{j}(x)(\partial_{k}u_{1}^{f})'(\sigma,x)(\partial_{l}u_{1}^{g})'(s-\sigma,x)\Big)\\
 &\ \ \ \ \ \ \ \qquad \qquad +2\sum_{j=1}^{n}\partial_{j}\Big(\sum_{k,l=1}^{n}c_{kl}^{j}(x)(\partial_{k}u_{1}^{g})'(s-\sigma,x)(\partial_{l}u_{1}^{f})'(\sigma,x)\Big)\Bigg]d\sigma\Bigg) dt.\\
 \end{aligned}
 \end{align*}
 Thus, $J_{3}(s,x)$ is given by 
 \begin{align*}
 \begin{aligned}
 J_{3}(s,x)&=2\sum_{j=1}^{n}\partial_{j}\Bigg[\sum_{k,l=1}^{n}c_{kl}^{j}(x)\partial_{k}u_{1}^{f}*\partial_{l}u_{1}^{g}(s,x)+\sum_{k,l=1}^{n}c_{kl}^{j}(x)\partial_{k}u_{1}^{g}*\partial_{l}u_{1}^{f}(s,x)\Bigg] -\int\limits_{0}^{s}Bu_{2}(t,x;s)dt\\
 &\ \ \ \  -\int\limits_{0}^{s}\Bigg(\int\limits_{0}^{t}\cos\left\{(t-\sigma)B^{1/2}\right\}\Bigg[2\sum_{j=1}^{n}\partial_{j}\Big(\sum_{k,l=1}^{n}c_{kl}^{j}(x)(\partial_{k}u_{1}^{f})'(\sigma,x)\partial_{l}u_{1}^{g}(s-\sigma,x)\Big)\\
 &\ \ \ \ \ \ \ \ \ \ \ \ \ \ \ \   \ +2\sum_{j=1}^{n}\partial_{j}\Big(\sum_{k,l=1}^{n}c_{kl}^{j}(x)\partial_{k}u_{1}^{g}(s-\sigma,x)(\partial_{l}u_{1}^{f})'(\sigma,x)\Big)\Bigg]d\sigma\Bigg) dt\\
 &\ \ -\int\limits_{0}^{s}\Bigg(B^{-1/2}\int\limits_{0}^{t}\sin\left\{(t-\sigma)B^{1/2}\right\}\Bigg[2\sum_{j=1}^{n}\partial_{j}\Big(\sum_{k,l=1}^{n}c_{kl}^{j}(x)(\partial_{k}u_{1}^{f})'(\sigma,x)(\partial_{l}u_{1}^{g})'(s-\sigma,x)\Big)\\
 &\ \ \ \ \ \ \ \qquad \qquad +2\sum_{j=1}^{n}\partial_{j}\Big(\sum_{k,l=1}^{n}c_{kl}^{j}(x)(\partial_{k}u_{1}^{g})'(s-\sigma,x)(\partial_{l}u_{1}^{f})'(\sigma,x)\Big)\Bigg]d\sigma\Bigg) dt.\\
 \end{aligned}
 \end{align*}
Hence $J_{1}(s,x)+J_{2}(s,x)+J_{3}(s,x)$ is given by 
 \begin{align*}
 \begin{aligned}
 &J_{1}(s,x)+J_{2}(s,x)+J_{3}(s,x)\\
 &=6\sum_{j=1}^{n}\int\limits_{0}^{s}\cos\left\{(s-\sigma)B^{1/2}\right\}\partial_{j}\Bigg[\sum_{k,l=1}^{n}c_{kl}^{j}(x)\left\{\partial_{k}u_{1}^{f}(\sigma,x)\partial_{l}u_{1}^{g}(s-\sigma,x)+\partial_{k}u_{1}^{g}(s-\sigma,x)\partial_{l}u_{1}^{f}(\sigma,x)
 \right\}\Bigg]d\sigma\\
 &\ +2\sum_{j=1}^{n}\partial_{j}\Bigg[\sum_{k,l=1}^{n}c_{kl}^{j}(x)\partial_{k}u_{1}^{f}*\partial_{l}u_{1}^{g}(s,x)+\sum_{k,l=1}^{n}c_{kl}^{j}(x)\partial_{k}u_{1}^{g}*\partial_{l}u_{1}^{f}(s,x)\Bigg] -\int\limits_{0}^{s}Bu_{2}(t,x;s)dt\\
 &\ -4\sum_{j=1}^{n}B^{-1/2}\int\limits_{0}^{s}\sin\left\{(s-\sigma)B^{1/2}\right\}\partial_{j}\Bigg[\sum_{k,l=1}^{n}c_{kl}^{j}(x)\Big\{(\partial_{k}u_{1}^{f})'(\sigma,x)\partial_{l}u_{1}^{g}(s-\sigma,x)\\
 &\ \ \ \ \ \ \ \ \ \ \quad \qquad \qquad \qquad \qquad \qquad \qquad \qquad \qquad \quad +\partial_{k}u_{1}^{g}(s-\sigma,x)(\partial_{l}u_{1}^{f})'(\sigma,x)\Big\}\Bigg]d\sigma\Bigg) dt\\
 &\ \ \   -2\sum_{j=1}^{n}\int\limits_{0}^{s}\Bigg(\int\limits_{0}^{t}\cos\left\{(t-\sigma)B^{1/2}\right\}\partial_{j}\Bigg[\sum_{k,l=1}^{n}c_{kl}^{j}(x)\Big\{(\partial_{k}u_{1}^{f})'(\sigma,x)\partial_{l}u_{1}^{g}(s-\sigma,x)\\
 &\ \ \ \ \ \ \ \ \ \ \quad \qquad \qquad \qquad \qquad \qquad \qquad \qquad \qquad \quad +\partial_{k}u_{1}^{g}(s-\sigma,x)(\partial_{l}u_{1}^{f})'(\sigma,x)\Big\}\Bigg]d\sigma\Bigg) dt\\
 &\ \ -2\sum_{j=1}^{n}\int\limits_{0}^{s}\Bigg(B^{-1/2}\int\limits_{0}^{t}\sin\left\{(t-\sigma)B^{1/2}\right\}\partial_{j}\Bigg[\Big(\sum_{k,l=1}^{n}c_{kl}^{j}(x)\Big\{(\partial_{k}u_{1}^{f})'(\sigma,x)(\partial_{l}u_{1}^{g})'(s-\sigma,x)\\
 &\qquad \qquad \qquad \qquad \qquad \qquad \qquad \qquad \qquad \qquad \ \ +(\partial_{k}u_{1}^{g})'(s-\sigma,x)(\partial_{l}u_{1}^{f})'(\sigma,x)\Big\}\Bigg]d\sigma\Bigg) dt.\\
 \end{aligned}
 \end{align*}
 Thus, we have the equation for $u_{2}^{(-1)}(s,x)$ is given by 
 \begin{align}\label{Equation for u2(-1)}
 \begin{aligned}
 \begin{cases}
 &\partial_{s}^{2}u_{2}^{(-1)}(s,x)-\nabla_{x}(\gamma(x)\nabla_{x}u_{2}^{(-1)}(s,x))\\
 & =2\sum_{j=1}^{n}\partial_{j}\Bigg[\sum_{k,l=1}^{n}c_{kl}^{j}(x)\Big\{\partial_{k}u_{1}^{f}*\partial_{l}u_{1}^{g}(s,x)+\partial_{k}u_{1}^{g}*\partial_{l}u_{1}^{f}(s,x)\Big\}\Bigg]\\
 &\ \ \ +I_{1}(s,x)+I_{2}(s,x)+I_{3}(s,x)+I_{4}(s,x), s\in\mathbb{R},\ x\in\Omega,\\
 &u_{2}^{(-1)}(0,x)=\partial_{s}u_{2}^{(-1)}(0,x)=0,\ x\in\Omega,\\
 &u_{2}^{(-1)}(s,x)|_{\partial Q}=0,
 \end{cases}
 \end{aligned}
 \end{align}
 where 
 \begin{align}\label{Definition of Ij's}
 \begin{aligned}
 &I_{1}(s,x):=6\sum_{j=1}^{n}\int\limits_{0}^{s}\cos\left\{(s-\sigma)B^{1/2}\right\}\partial_{j}\Bigg[\sum_{k,l=1}^{n}c_{kl}^{j}(x)\left\{\partial_{k}u_{1}^{f}(\sigma,x)\partial_{l}u_{1}^{g}(s-\sigma,x)+\partial_{k}u_{1}^{g}(s-\sigma,x)\partial_{l}u_{1}^{f}(\sigma,x)
 \right\}\Bigg]d\sigma,\\
 &I_{2}(s,x):=-4\sum_{j=1}^{n}B^{-1/2}\int\limits_{0}^{s}\sin\left\{(s-\sigma)B^{1/2}\right\}\partial_{j}\Bigg[\sum_{k,l=1}^{n}c_{kl}^{j}(x)\Big\{(\partial_{k}u_{1}^{f})'(\sigma,x)\partial_{l}u_{1}^{g}(s-\sigma,x)\\
 &\ \ \ \ \ \ \ \ \ \ \quad \qquad \qquad \qquad \qquad \qquad \qquad \qquad \qquad \quad +\partial_{k}u_{1}^{g}(s-\sigma,x)(\partial_{l}u_{1}^{f})'(\sigma,x)\Big\}\Bigg]d\sigma\Bigg) dt,\\
 &I_{3}(s,x):=-2\sum_{j=1}^{n}\int\limits_{0}^{s}\Bigg(\int\limits_{0}^{t}\cos\left\{(t-\sigma)B^{1/2}\right\}\partial_{j}\Bigg[\sum_{k,l=1}^{n}c_{kl}^{j}(x)\Big\{(\partial_{k}u_{1}^{f})'(\sigma,x)\partial_{l}u_{1}^{g}(s-\sigma,x)\\
 &\ \ \ \ \ \ \ \ \ \ \quad \qquad \qquad \qquad \qquad \qquad \qquad \qquad \qquad \quad +\partial_{k}u_{1}^{g}(s-\sigma,x)(\partial_{l}u_{1}^{f})'(\sigma,x)\Big\}\Bigg]d\sigma\Bigg) dt,\\
 &I_{4}(s,x):= -2\sum_{j=1}^{n}\int\limits_{0}^{s}\Bigg(B^{-1/2}\int\limits_{0}^{t}\sin\left\{(t-\sigma)B^{1/2}\right\}\partial_{j}\Bigg[\Big(\sum_{k,l=1}^{n}c_{kl}^{j}(x)\Big\{(\partial_{k}u_{1}^{f})'(\sigma,x)(\partial_{l}u_{1}^{g})'(s-\sigma,x)\\
 &\qquad \qquad \qquad \qquad \qquad \qquad \qquad \qquad \qquad \qquad \ \ +(\partial_{k}u_{1}^{g})'(s-\sigma,x)(\partial_{l}u_{1}^{f})'(\sigma,x)\Big\}\Bigg]d\sigma\Bigg) dt.\\
 \end{aligned}
 \end{align}
 Also the corresponding Neumann data is 
 \begin{align}\label{Neumann data}
 \begin{aligned}
 &\gamma(x)\partial_{\nu}u_{2}^{(-1)}(t,x)+2\sum_{j=1}^{n}\nu_{j}(x)\Big(\sum_{k,l=1}^{n}c_{kl}^{j}(x)\partial_{k}u_{1}^{f}*\partial_{l}u_{1}^{g}(t,x)\Big)\\
 &+2\sum_{j=1}^{n}\nu_{j}(x)\Big(\sum_{k,l=1}^{n}c_{kl}^{j}(x)\partial_{k}u_{1}^{g}*\partial_{l}u_{1}^{f}(t,x)\Big)=0,\ \  (t,x)\in {\mathbb R}\times\partial\Omega.
 \end{aligned}
 \end{align}
 Furthermore the approximation of each of $I_{j}(s,x)'s$ is given by 
 \begin{align}\label{Approximation for I1}
 \begin{aligned}
 &I_{1}(s,x)\approx 6\sum_{j=1}^{n}\int\limits_{0}^{s}\partial_{j}\Bigg[\sum_{k,l=1}^{n}c_{kl}^{j}(x)\left\{\partial_{k}u_{1}^{f}(\sigma,x)\partial_{l}u_{1}^{g}(s-\sigma,x)+\partial_{k}u_{1}^{g}(s-\sigma,x)\partial_{l}u_{1}^{f}(\sigma,x)
 \right\}\Bigg]d\sigma\\
 &\ \ \ =6\sum_{j=1}^{n}\partial_{j}\Bigg[\sum_{k,l=1}^{n}c_{kl}^{j}(x)\Big\{\partial_{k}u_{1}^{f}*\partial_{l}u_{1}^{g}(s,x)+\partial_{k}u_{1}^{g}*\partial_{l}u_{1}^{f}(s,x)\Big\}\Bigg]
 \end{aligned}
 \end{align}

 \begin{align}\label{Approximation for I2}
 \begin{aligned}
 &I_{2}(s,x)\approx -4\sum_{j=1}^{n}\int\limits_{0}^{s}(s-\sigma)\partial_{j}\Bigg[\sum_{k,l=1}^{n}c_{kl}^{j}(x)\Big\{(\partial_{k}u_{1}^{f})'(\sigma,x)\partial_{l}u_{1}^{g}(s-\sigma,x) +\partial_{k}u_{1}^{g}(s-\sigma,x)(\partial_{l}u_{1}^{f})'(\sigma,x)\Big\}\Bigg]d\sigma\\
 &\ \ =-4\sum_{j=1}^{n}\partial_{j}\left[\sum_{k,l=1}^{n}c_{kl}^{j}(x)\left\{\left((\partial_{k}u_{1}^{f})'*F_{l}\right)(s,x)+\left(F_{k}*(\partial_{l}u_{1}^{f})'\right)(s,x)\right\}\right],
 \end{aligned}
 \end{align}
 where \[F_{k}(s,x):=s\partial_{k}u_{1}^{g}(s,x),\,\, F_{l}(s,x):= s\partial_{l}u_{1}^{g}(s,x).\]
 Here for instance \eqref{Approximation for I2} has been derived as follows. Observe that
 \begin{equation}\label{derivation of dominant part of I_2}
 \begin{array}{ll}
 B^{-1/2}\sin\left\{(s-\sigma)B^{1/2}\right\}
 =\int_0^\infty \lambda^{-1/2}\sin\left\{\lambda^{1/2}(s-\sigma)\right\}dE(\lambda)\\
 =(s-\sigma)I-(s-\sigma)^2\int_0^\infty\lambda^{1/2}\big(\int_0^1(1-\theta)\sin\{\theta(s-\sigma)\lambda^{1/2}\}\,d\theta\big)\,dE(\lambda),
 \end{array}
 \end{equation}
 where $\{E(\lambda)\}_{\lambda>0}$ are projection operators on $L^2(\Omega)$ giving the spectral decomposition
 $B=\int_0^\infty \lambda dE(\lambda)$ of $B$. Then we have just taken the first term of \eqref{derivation of dominant part of I_2}. Of course we have to put the meaning to this approximation. We will give that later in Appendix by taking the Laplace transform. More precisely we will show that the Laplace transform of $I_2(s,x)$ has its dominant part given as the Laplace transform of the dominant part given by \eqref{Approximation for I2}. 
 \begin{align*}
 \begin{aligned}
 &I_{3}(s,x)\approx -2\sum_{j=1}^{n}\int\limits_{0}^{s}\Bigg(\int\limits_{0}^{t}\partial_{j}\Bigg[\sum_{k,l=1}^{n}c_{kl}^{j}(x)\Big\{(\partial_{k}u_{1}^{f})'(\sigma,x)\partial_{l}u_{1}^{g}(s-\sigma,x) +\partial_{k}u_{1}^{g}(s-\sigma,x)(\partial_{l}u_{1}^{f})'(\sigma,x)\Big\}\Bigg]d\sigma\Bigg) dt\\
 &\ = -2\sum_{j=1}^{n}\int\limits_{\sigma=0}^{\sigma=s}\Bigg(\int\limits_{t=\sigma}^{t=s}\partial_{j}\Bigg[\sum_{k,l=1}^{n}c_{kl}^{j}(x)\Big\{(\partial_{k}u_{1}^{f})'(\sigma,x)\partial_{l}u_{1}^{g}(s-\sigma,x) +\partial_{k}u_{1}^{g}(s-\sigma,x)(\partial_{l}u_{1}^{f})'(\sigma,x)\Big\}\Bigg]dt\Bigg) d\sigma\\
 &\ =-2\sum_{j=1}^{n}\int\limits_{\sigma=0}^{\sigma=s}(s-\sigma)\partial_{j}\Bigg[\sum_{k,l=1}^{n}c_{kl}^{j}(x)\Big\{(\partial_{k}u_{1}^{f})'(\sigma,x)\partial_{l}u_{1}^{g}(s-\sigma,x) +\partial_{k}u_{1}^{g}(s-\sigma,x)(\partial_{l}u_{1}^{f})'(\sigma,x)\Big\}\Bigg] d\sigma\\
 &\ = -2\sum_{j=1}^{n}\partial_{j}\left[\sum_{k,l=1}^{n}c_{kl}^{j}(x)\left\{((\partial_{k}u_{1}^{f})'*F_{l})(s,x)+(F_{k}*(\partial_{l}u_{1}^{f})')(s,x)\right\}\right].
 \end{aligned}
 \end{align*}
 Thus, we have 
 \begin{align}\label{Approximation for I3}
 I_{3}(s,x)\approx -2\sum_{j=1}^{n}\partial_{j}\left[\sum_{k,l=1}^{n}c_{kl}^{j}(x)\left\{\left((\partial_{k}u_{1}^{f})'*F_{l}\right)(s,x)+\left(F_{k}*(\partial_{l}u_{1}^{f})'\right)(s,x)\right\}\right]
 \end{align}
 where 
 \[ F_{k}(s,x):=s\partial_{k}u_{1}^{g}(s,x),\,\, F_{l}(s,x):= s\partial_{l}u_{1}^{g}(s,x).\]
 Further for $I_{4}(s,x)$, we first integrate by parts to have 
 \begin{align*}
 \begin{aligned}
 &I_{4}(s,x)\approx -2\sum_{j=1}^{n}\int\limits_{0}^{s}\int\limits_{0}^{t}(t-\sigma)\partial_{j}\Bigg[\sum_{k,l=1}^{n}c_{kl}^{j}(x)\Big\{(\partial_{k}u_{1}^{f})'(\sigma,x)(\partial_{l}u_{1}^{g})'(s-\sigma,x)\\
 &\qquad \qquad \qquad \qquad \qquad \qquad \qquad \qquad \quad\qquad +(\partial_{k}u_{1}^{g})'(s-\sigma,x)(\partial_{l}u_{1}^{f})'(\sigma,x)\Big\}\Bigg]d\sigma dt\\
 &\ \ \ \ = -2\sum_{j=1}^{n}\int\limits_{\sigma=0}^{\sigma=s}\int\limits_{t=\sigma}^{t=s}(t-\sigma)\partial_{j}\Bigg[\sum_{k,l=1}^{n}c_{kl}^{j}(x)\Big\{(\partial_{k}u_{1}^{f})'(\sigma,x)(\partial_{l}u_{1}^{g})'(s-\sigma,x)\\ 
 &\qquad \qquad \qquad \qquad \qquad \qquad \qquad \qquad \quad\qquad +(\partial_{k}u_{1}^{g})'(s-\sigma,x)(\partial_{l}u_{1}^{f})'(\sigma,x)\Big\}\Bigg]dt d\sigma\\
 &\ \ \ \ =-\sum_{j=1}^{n}\int\limits_{0}^{s}(s-\sigma)^{2}\partial_{j}\Bigg[\sum_{k,l=1}^{n}c_{kl}^{j}(x)\Big\{(\partial_{k}u_{1}^{f})'(\sigma,x)(\partial_{l}u_{1}^{g})'(s-\sigma,x)\\ 
 &\qquad \qquad \qquad \qquad \qquad \qquad \qquad \qquad \quad\qquad +(\partial_{k}u_{1}^{g})'(s-\sigma,x)(\partial_{l}u_{1}^{f})'(\sigma,x)\Big\}\Bigg] d\sigma.\\
 \end{aligned}
 \end{align*}
 Then we have 
 \begin{align}\label{Approximation for I4}
 I_{4}(s,x)\approx -\sum_{j=1}^{n}\partial_{j}\left[ \sum_{k,l=1}^{n}\left\{\left((\partial_{k}u_{1}^{f})'*G_{l}\right)(s,x)+\left(G_{k}*(\partial_{k}u_{1}^{f})'\right)(s,x)\right\}\right],
 \end{align}
 where 
 \[G_{k}(s,x):= s^{2}(\partial_{k}u_{1}^{g})'(s,x),\,\, G_{l}(s,x):= s^{2}(\partial_{l}u_{1}^{g})'(s,x) .\]

By the standard estimate for the initial boundary value problem for second order hyperbolic equations of divergence form and an argument similar to derive the estimate \eqref{spectral estimate} to handle $I_j,\,j=1,2,3,4$, we have  \begin{equation}\label{estimate of u_2(-1)}
 	\Vert u^{(-1)}_2(s)\Vert_{H^1(\Omega)}=O(|s|^2),\, |s|\gg1. 
\end{equation}
 
 Next we define the Laplace transform of a function $f$ satisfying \eqref{estimate of u_2(-1)} by 
 \[\widehat{f}(\tau,x):=\int\limits_{0}^{\infty}e^{-t\tau}f(t,x)\D x\]
which is well-defined for all $\tau=\tau_{R}+i\tau_{I}$ with $\tau_{R}\geq 1$. 
  Further, for $\tau\in \mathbb{C}$ defined as above, we have $\widehat{u_{2}^{(-1)}}(\tau,x)$ and $\widehat{\partial_{k}u_{1}^{f}}(\tau,x)$, $\widehat{\partial_{j}u_{1}^{g}}(\tau,x)$ are well-defined because of the estimate \eqref{estimate of u_2(-1)} and standard energy estimate for $u_{1}^{f\ \mbox{or}\  g}$ respectively.  Now for each fixed $\tau$, we have $\widehat{u_{2}^{(-1)}}(\tau,x)$ is the solution to the following boundary value problem:
 \begin{align}\label{equation for widehat u2 (-1)}
 \begin{aligned}
 \begin{cases}
 &\tau^{2}\widehat{u^{(-1)}_{2}}(\tau,x)-\nabla_{x}\cdot(\gamma(x)\nabla_{x}\widehat{u_{2}^{(-1)}})(\tau,x)\\
 &\ \ \  = 2\sum_{j=1}^{n}\partial_{j}\Bigg[\sum_{k,l=1}^{n}c_{kl}^{j}(x)\left\{\partial_{k}\widehat{u_{1}^{f}}(\tau,x)\partial_{l}\widehat{u_{1}^{g}}(\tau,x)+\partial_{k}\widehat{u_{1}^{g}}(\tau,x)\partial_{l}\widehat{u_{1}^{f}}(\tau,x)\right\}\Bigg]\\
 &\qquad\qquad +\widehat{I_{1}}(\tau,x)+\widehat{I_{2}}(\tau,x)+\widehat{I_{3}}(\tau,x)+\widehat{I_{4}}(\tau,x),\ \  x\in\Omega,\\
 &\widehat{u^{(-1)}_{2}}(\tau,x)=0,\ \ x\in\partial\Omega.
 \end{cases}
 \end{aligned}
 \end{align}
 Also, from \eqref{Neumann data} we have
 \begin{align}\label{equality of Neumann data}
 \begin{aligned}
 &\gamma(x)\partial_{\nu}\widehat{u_{2}^{(-1)}}(\tau,x)+2\sum_{j=1}^{n}\nu_{j}(x)\Big(\sum_{k,l=1}^{n}c_{kl}^{j}(x)\partial_{k}\widehat{u_{1}^{f}}(\tau,x)\partial_{l}\widehat{u_{1}^{g}}(\tau,x)\Big)\\
 &+2\sum_{j=1}^{n}\nu_{j}(x)\Big(\sum_{k,l=1}^{n}c_{kl}^{j}(x)\partial_{k}\widehat{u_{1}^{g}}(\tau,x)\partial_{l}\widehat{u_{1}^{f}}(\tau,x)\Big)=0, \ \ x\in\partial\Omega.
 \end{aligned}
 \end{align}
 
 Next take a solution $\widehat{w}(\tau,x)$ of the following equation
 \begin{align}\label{equation for v}
 \tau^{2}v-\nabla_{x}\cdot(\gamma(x)\nabla_{x}v)=0,\ \ \ x \in\Omega.
 \end{align}
 Multiplying the first equation of $\eqref{equation for widehat u2 (-1)}$ by $\widehat{w}(\tau,x)$ and integrating over $\Omega$, we have
\begin{align*}
 \begin{aligned}
 &\int\limits_{\Omega}\tau^{2}\widehat{u^{(-1)}_{2}}(\tau,x)\widehat{w}(\tau,x)dx-\int\limits_{\Omega}\nabla_{x}\cdot(\gamma(x)\nabla_{x}\widehat{u_{2}^{(-1)}})(\tau,x)\widehat{w}(\tau,x)dx\\ &\ \ \ \  =2\sum_{j=1}^{n}\sum_{k,l=1}^{n}\int\limits_{\Omega}\partial_{j}\Big(c_{kl}^{j}(x)\partial_{k}\widehat{u_{1}^{f}}(\tau,x)\partial_{l}\widehat{u_{1}^{g}}(\tau,x)\Big)\widehat{w}(\tau,x)dx\\
 &\ \ \ \ \ \ +2\sum_{j=1}^{n}\sum_{k,l=1}^{n}\int\limits_{\Omega}\partial_{j}\Big(c_{kl}^{j}(x)\partial_{k}\widehat{u_{1}^{g}}(\tau,x)\partial_{l}\widehat{u_{1}^{f}}(\tau,x)\Big)\widehat{w}(\tau,x)dx+\int\limits_{\Omega}\widehat{I_{1}}(\tau,x)\widehat{w}(\tau,x)\D x\\
 & \ \ \ \ \ \ + \int\limits_{\Omega}\widehat{I_{2}}(\tau,x)\widehat{w}(\tau,x)\D x +\int\limits_{\Omega}\widehat{I_{3}}(\tau,x)\widehat{w}(\tau,x)\D x+\int\limits_{\Omega}\widehat{I_{4}}(\tau,x)\widehat{w}(\tau,x)\D x.
 \end{aligned}
 \end{align*}
 Then, integration by parts and using the fact that $\widehat{w}(\tau,x)$ is a solution of $\eqref{equation for v}$, we have
 \begin{align*}
 \begin{aligned}
 &-\int\limits_{\partial\Omega}\gamma(x)\partial_{\nu}\widehat{u_{2}^{(-1)}}(\tau,x)\widehat{w}(\tau,x)dS_{x}+\int\limits_{\partial\Omega}\gamma(x)\widehat{u_{2}^{(-1)}}(\tau,x)\partial_{\nu}\widehat{w}(\tau,x)dS_{x}=\\
 &\ \ \ -2\sum_{j=1}^{n}\sum_{k,l=1}^{n}\int\limits_{\Omega}\Big(c_{kl}^{j}(x)\partial_{k}\widehat{u_{1}^{f}}(\tau,x)\partial_{l}\widehat{u_{1}^{g}}(\tau,x)\Big)\partial_{j}\widehat{w}(\tau,x)dx\\
 &\ \ \ -2\sum_{j=1}^{n}\sum_{k,l=1}^{n}\int\limits_{\Omega}\Big(c_{kl}^{j}(x)\partial_{k}\widehat{u_{1}^{g}}(\tau,x)\partial_{l}\widehat{u_{1}^{f}}(\tau,x)\Big)\partial_{j}\widehat{w}(\tau,x)dx\\
 &\ \ \ +2\sum_{j=1}^{n}\sum_{k,l=1}^{n}\int\limits_{\partial\Omega}\nu_{j}(x)\Big(c_{kl}^{j}(x)\partial_{k}\widehat{u_{1}^{f}}(\tau,x)\partial_{l}\widehat{u_{1}^{g}}(\tau,x)\Big)\widehat{w}(\tau,x)dS_{x}\\
 &\ \ \ +2\sum_{j=1}^{n}\sum_{k,l=1}^{n}\int\limits_{\partial\Omega}\nu_{j}(x)\Big(c_{kl}^{j}(x)\partial_{k}\widehat{u_{1}^{g}}(\tau,x)\partial_{l}\widehat{u_{1}^{f}}(\tau,x)\Big)\widehat{w}(\tau,x)dS_{x} +\int\limits_{\Omega}\widehat{I_{1}}(\tau,x)\widehat{w}(\tau,x)\D x\\
 & \ \ \  + \int\limits_{\Omega}\widehat{I_{2}}(\tau,x)\widehat{w}(\tau,x)\D x +\int\limits_{\Omega}\widehat{I_{3}}(\tau,x)\widehat{w}(\tau,x)\D x+\int\limits_{\Omega}\widehat{I_{4}}(\tau,x)\widehat{w}(\tau,x)\D x.
 \end{aligned}
 \end{align*}
 Using \eqref{equality of Neumann data}, this implies
 \begin{align}\label{integral identity}
 \begin{aligned}
 &-2\sum_{j=1}^{n}\sum_{k,l=1}^{n}\int\limits_{\Omega}c_{kl}^{j}(x)\Big\{\partial_{k}\widehat{u_{1}^{f}}(\tau,x)\partial_{l}\widehat{u_{1}^{g}}(\tau,x)+\partial_{k}\widehat{u_{1}^{g}}(\tau,x)\partial_{l}\widehat{u_{1}^{f}}(\tau,x)\Big\}\partial_{j}\widehat{w}(\tau,x)dx\\
 & \ \ \ +\int\limits_{\Omega}\widehat{I_{1}}(\tau,x)\widehat{w}(\tau,x)\D x + \int\limits_{\Omega}\widehat{I_{2}}(\tau,x)\widehat{w}(\tau,x)\D x +\int\limits_{\Omega}\widehat{I_{3}}(\tau,x)\widehat{w}(\tau,x)\D x+\int\limits_{\Omega}\widehat{I_{4}}(\tau,x)\widehat{w}(\tau,x)\D x=0,
 \end{aligned}
 \end{align}
 where $\widehat{u_{1}^{f}}(\tau,x)$ and  $\widehat{u_{1}^{g}}(\tau,x)$ are solution to the following equation 
 \begin{align}\label{equation for widehat u}
 \tau^{2}u-\nabla_{x}\cdot(\gamma(x)\nabla_{x}u)=0\, \ \text{in} \ \ \Omega
 \end{align}
 with the Dirichlet data on $\partial\Omega$ equal to $f$ and $g$, respectively. Here note that by Appendix C, each term of the left hand side of this equation is analytic in $S^i$, where $S^i$ is the interior of $S:=\{\tau=\tau_R+\sqrt{-1}\,\tau_I\in\mathbb{C}: \tau_R,\,\tau_I\in\mathbb{R},\, \tau_R\ge1,\,|\tau_I|<\kappa|\tau_R|\}$ with a fixed small $\kappa>0$. By using Lemma \ref{Useful lemma} and Appendix D, we have 
 \begin{align*}
 \begin{aligned}
 &-2\sum_{j=1}^{n}\sum_{k,l=1}^{n}\int\limits_{\Omega}c_{kl}^{j}(x)\Big\{\partial_{k}\widehat{u_{1}^{f}}(\tau,x)\partial_{l}\widehat{u_{1}^{g}}(\tau,x)+\partial_{k}\widehat{u_{1}^{g}}(\tau,x)\partial_{l}\widehat{u_{1}^{f}}(\tau,x)\Big\}\partial_{j}\widehat{w}(\tau,x)dx\\
 &-6\sum_{j=1}^{n}\sum_{k,l=1}^{n}c_{kl}^{j}(x)\Big\{\partial_{k}\widehat{u_{1}^{f}}(\tau,x)\partial_{l}\widehat{u_{1}^{g}}(\tau,x)+\partial_{k}\widehat{u_{1}^{g}}(\tau,x)\partial_{l}\widehat{u_{1}^{f}}(\tau,x)\Big\}\PD_{j}\widehat{w}(\tau,x)\D x\\
 &+4\sum_{j=1}^{n}\sum_{k,l=1}^{n}\int\limits_{\Omega}c_{kl}^{j}(x)\left\{\widehat{(\partial_{k}u_{1}^{f})'}(\tau,x)\widehat{F_{l}}(\tau,x)+\widehat{F_{k}}\widehat{(\partial_{l}u_{1}^{f})'}(\tau,x)\right\}\partial_{j}\widehat{w}(\tau,x)\D x\\
 &+2\sum_{j=1}^{n}\sum_{k,l=1}^{n}\int\limits_{\Omega}c_{kl}^{j}(x)\left\{\widehat{(\partial_{k}u_{1}^{f})'}(\tau,x)\widehat{F_{l})}(\tau,x)+\widehat{F_{k}}(\tau,x)\widehat{(\partial_{l}u_{1}^{f})'}(\tau,x)\right\}\partial_{j}\widehat{w}(\tau,x)\D x\\
 &+\sum_{j=1}^{n}\sum_{k,l=1}^{n}\int\limits_{\Omega} \left\{\widehat{(\partial_{k}u_{1}^{f})'}(\tau,x)\widehat{G_{l}}(\tau,x)+\widehat{G_{k}}(\tau,x)\widehat{(\partial_{k}u_{1}^{f})'}(\tau,x)\right\}\partial_{j}\widehat{w}(\tau,x)\D x=0
 \end{aligned}
 \end{align*}
 holds for $\tau\in S^i$.
 This will give us 
 	\begin{align}\label{Integral identity after using useful lemma}
 	\begin{aligned}
 	 &-8\sum_{j=1}^{n}\sum_{k,l=1}^{n}c_{kl}^{j}(x)\Big\{\partial_{k}\widehat{u_{1}^{f}}(\tau,x)\partial_{l}\widehat{u_{1}^{g}}(\tau,x)+\partial_{k}\widehat{u_{1}^{g}}(\tau,x)\partial_{l}\widehat{u_{1}^{f}}(\tau,x)\Big\}\PD_{j}\widehat{w}(\tau,x)\D x\\
 	&+4\sum_{j=1}^{n}\sum_{k,l=1}^{n}\int\limits_{\Omega}c_{kl}^{j}(x)\left\{\widehat{(\partial_{k}u_{1}^{f})'}(\tau,x)\widehat{F_{l}}(\tau,x)+\widehat{F_{k}}\widehat{(\partial_{l}u_{1}^{f})'}(\tau,x)\right\}\partial_{j}\widehat{w}(\tau,x)\D x\\
 	&+2\sum_{j=1}^{n}\sum_{k,l=1}^{n}\int\limits_{\Omega}c_{kl}^{j}(x)\left\{\widehat{(\partial_{k}u_{1}^{f})'}(\tau,x)\widehat{F_{l})}(\tau,x)+\widehat{F_{k}}(\tau,x)\widehat{(\partial_{l}u_{1}^{f})'}(\tau,x)\right\}\partial_{j}\widehat{w}(\tau,x)\D x\\
 	&+\sum_{j=1}^{n}\sum_{k,l=1}^{n}\int\limits_{\Omega} \left\{\widehat{(\partial_{k}u_{1}^{f})'}(\tau,x)\widehat{G_{l}}(\tau,x)+\widehat{G_{k}}(\tau,x)\widehat{(\partial_{k}u_{1}^{f})'}(\tau,x)\right\}\partial_{j}\widehat{w}(\tau,x)\D x=0.
 	\end{aligned}
 	\end{align}
 The above identity holds for any solution  $\widehat{w}(\tau,x)$ to \eqref{equation for widehat u} and any $\tau\in S^i$.
 
 Now we want to use complex geometric optics solutions (CGO solutions) for $\widehat{u^{f}_{1}}(\cdot,\tau)$, $\widehat{u_{1}^{g}}(\cdot,\tau)$ for a fixed $\tau>1$ to derive $c^{j}_{kl}(x)=0$, for all $1\leq j,k,l\leq n$ and $x\in\Omega$.  
 
 Let $u$ be the solution to
 \begin{equation}\label{standard linear equation}
 \left\{\begin{array}{ll}
 \partial_t^2 u-\nabla\cdot(\gamma\nabla u)=0\,\,&\text{in $Q_T$},\\
 u=\chi(t)\widetilde{f}(x)\,\,&\text{on $\partial Q_T$},\\
 u(0,x)=\PD_{t}u(0,x)=0\,\,& \mbox{in} \ \  \Omega.
 \end{array}
 \right.
 \end{equation}
 Consider the even extension of this $u$ corresponding to the even extension of the above $f=\chi(t)\widetilde{f}$. By abusing the notation, we denote the extended $u$ by the same notation $u$. Further extend this $u$ to the whole ${\mathbb R}\times\Omega$ and denote the extended one by the same notation $u$. Then $\widehat u=\widehat u(\tau, \cdot)$ is the solution to the following boundary value problem:
 \begin{equation}\label{Equation for widehat u}
 \left\{
 \begin{array}{ll}
 \tau^2 \widehat u-\nabla\cdot(\gamma\nabla\widehat u)=0\,\,&\text{in $\Omega$},\\
 \widehat u=\widehat\chi(\tau)\widetilde{f}\,\,&\text{on $\partial\Omega$}.
 \end{array}
 \right.
 \end{equation}
 If we fix $\tau>0$ large enough, then note that $\widehat\chi(\tau)>0$. Hence  $\widehat\chi(\tau)\widetilde f$ can be taken arbitrarily due to the freedom of choosing $\widetilde f$. Hence we can just look for some special solution $\widehat u$ of $\tau^2\widehat u-\nabla\cdot(\gamma\nabla_{x}\widehat u)=0$ in ${\mathbb R}^n$ which would be the CGO solution. One might concern about the freedom of choosing $f=\chi(t)\widetilde f$ in the argument given here, because the original $f$ should be coming from that in \eqref{equation of interest}. As already mentioned before on the unique solvability of \eqref{equation of interest},
 recall that $B_M^0$ can be any fixed single $f\in C_0^\infty(\partial Q_T)$ and can consider the
 $\epsilon$-expansion of the solution $u$ to \eqref{equation of interest}with this Dirichlet data $\epsilon f,\,0<\epsilon<\epsilon_0$ to derive \eqref{standard linear equation}.
 
 To have the CGO solutions  of \eqref{equation for widehat u} with fixed $\tau>0$ large enough, we will make use of the following theorem given by Sylvester and Uhlmann.
 \begin{theorem}[See Theorem $1.1$ in \cite{Sylvester Uhlmann}]
 	Let $\Omega\subset\mathbb{R}^{n}$ be  a bounded domain and $s>n/2$,  there exists a constant  $C(\tau)$ such that if 
 	\begin{align*}
 	\zeta\cdot\zeta=0, \ \ \zeta\in\mathbb{C}^{n}
 	\end{align*}
 	and $|\zeta|$ is large, then there exists $u(x,\zeta;\tau)$ solving to $\eqref{equation for widehat u}$ such that
 	\begin{align*}
 	u(x,\zeta;\tau):= e^{\zeta\cdot x}m(x)\lb 1+R(x,\zeta;\tau)\rb,
 	\end{align*}
 	where $m(x)=\gamma^{-\frac{1}{2}}(x)$ and
 	\begin{align}\label{Hs estimate on R}
 	   \lVert R\rVert_{H^{s}(\Omega)}\leq \frac{C}{\lvert\zeta\rvert}
 	\end{align}
 	and it also satisfies 
 	\begin{align*}
 \lVert R\rVert_{L^{2}(\Omega)}\leq \frac{C(\tau)}{\lvert\zeta\rvert},\ \lVert\nabla_{x}R\rVert_{L^{2}(\Omega)}	\leq C(\tau)
 	\end{align*}
 	where $C(\tau)$ is a constant depending only on $\tau, \gamma$ and $\Omega$.
 \end{theorem}
 Using the above theorem, we have for fixed $\tau$, the expressions for $\widehat{u_{1}^{f}}(\tau,x)$ and $\widehat{u_{1}^{g}}(\tau,x)$ given by 
 \begin{align}\label{CGO to widehat u}
 \left\{\begin{aligned}
 &\widehat{u_{1}^{f}}(x,\zeta_{1};\tau):= e^{\zeta_{1}\cdot x}m(x)(1+R_{1}(x,\zeta_{1},\tau)),\\
 &\widehat{u_{1}^{g}}(x,\zeta_{2},\tau):=e^{\zeta_{2}\cdot x}m(x)(1+R_{2}(x,\zeta_{2},\tau)),
 \end{aligned}
 \right.
 \end{align}
 where $m(x)=\gamma(x)^{-\frac{1}{2}}$ and $R_{i},\,i=1,2$ satisfy the following estimate
 \begin{align}\label{estimate on error term in CGO}
 \begin{aligned}
 \lVert R_{i}\rVert_{L^{2}(\Omega)}\leq \frac{C(\tau)}{\lvert\zeta\rvert},\ \lVert\nabla_{x}R_{i}\rVert_{L^{2}(\Omega)}	\leq C(\tau)\ \ \text{for i=1,\,2}
 \end{aligned}
 \end{align}
 with $C(\tau)$ likewise before.
 
 \medskip
 By direct computations we have
  	\begin{align*}
 	\begin{aligned}
 	&\widehat{\partial_{k}u_{1}^{f}}(\tau,x)= e^{\zeta_{1}\cdot x}\Big( \zeta_{1}^{k}m(x)+\partial_{k}m(x)+\zeta_{1}^{k}m(x)R_{1}(x,\zeta_{1};\tau)+\partial_{k}m(x)R_{1}(x,\zeta_{1};\tau)+m(x)\partial_{k}R_{1}(x,\zeta_{1},\tau)\Big),\\
 	&\widehat{\partial_{l}u_{1}^{g}}(\tau,x)= e^{\zeta_{2}\cdot x}\Big( \zeta_{2}^{l}m(x)+\partial_{l}m(x)+\zeta_{2}^{l}m(x)R_{2}(x,\zeta_{2};\tau)+\partial_{l}m(x)R_{2}(x,\zeta_{2};\tau)+m(x)\partial_{l}R_{2}(x,\zeta_{2},\tau)\Big),\\
 	& \widehat{\partial_{k}u_{1}^{f}}'(\tau,x)=e^{\zeta_{1}\cdot x}\Big(\zeta_{1}^{k}m(x)R_{1}'(x,\zeta_{1};\tau)+\partial_{k}m(x)R_{1}'(x,\zeta_{1};\tau)+m(x)\partial_{k}R_{1}'(x,\zeta_{1},\tau)\Big),\\
 	&\widehat{\partial_{l}u_{1}^{g}}'(\tau,x)=e^{\zeta_{2}\cdot x}\Big(\zeta_{2}^{l}m(x)R_{2}'(x,\zeta_{2};\tau)+\partial_{l}m(x)R_{2}'(x,\zeta_{2};\tau)+m(x)\partial_{l}R_{2}'(x,\zeta_{2},\tau)\Big).
 	\end{aligned}
 	\end{align*}

\medskip\noindent
	\begin{remark} Concerning the derivative of $R_i,\,i=1,2$ with respect to $\tau$, we can argue as follows.
		Let $\widehat{u}$ solve \eqref{Equation for widehat u}, now define $\widehat{v}$ by 
	$\widehat{v}(\tau,x):={\gamma^{-\frac{1}{2}}(x)}\widehat{u}(\tau,x)$, then $\widehat{v}(\tau,x)$ will satisfies 
\[
-\Delta \widehat{v}(\tau,x)+q(\tau,x)\widehat{v}(\tau,x)=0,\,\, x\in \Omega\]
where \[q(\tau,x):=\frac{-\Delta\sqrt{\gamma(x)}}{\sqrt{\gamma(x)}}+\frac{\tau^{2}}{\gamma(x)}.\]
Then it is well known from \cite{Sylvester Uhlmann} that above equation has CGO solutions taking the form
\[\widehat{v}(\tau,x)=e^{\zeta\cdot x}\lb 1+R(x,\zeta,\tau)\rb,\]
where $\zeta\cdot\zeta=0$ and $R(x,\zeta,\tau)$ will be a solution to 
\begin{align}\label{Equation for remainder term CGO}
-\Delta R(x,\zeta,\tau)+2\zeta\cdot\nabla_{x}R(x,\zeta,\tau)+q(\tau,x)R(x,\zeta,\tau)=-q(\tau,x)
\end{align}
and it satisfies the estimate 
\begin{align*}
\lVert R\rVert_{L^{2}(\Omega)}\leq \frac{C}{\lvert\zeta\rvert}\ \mbox{and} \ \lVert \nabla_{x}R\rVert_{L^{2}(\Omega)}\leq C,
\end{align*}
where the positive constant $C$ depends only on $q(\tau,x)$. Now differentiating \eqref{Equation for remainder term CGO} with respect to $\tau$, we have 
\begin{align}\label{Equation for tau derivative of remainder term}
-\Delta R'(x,\zeta,\tau)+2\zeta\cdot\nabla_{x}R'(x,\zeta,\tau)+q(\tau,x)R'(x,\zeta,\tau)=-q(\tau,x)-q'(\tau,x)R(x,\zeta,\tau):=F(\tau,\zeta,x),
\end{align}
where we have abused the notation $'$ to denote the derivative with respect to  $\tau$. Now since $F\in L^{2}(\Omega)$ and $\lVert F\rVert_{L^{2}(\Omega)}\leq C$ for some constant depending only on $q(\tau,x)$,  therefore $R'(x,\zeta,\tau)$ will have the same estimate as that of $R(x,\zeta,\tau)$. Hence we have the following estimate 
	\begin{align}\label{Estimate for tau derivative of error term in CGO}
	\lVert R'\rVert_{L^{2}(\Omega)}\leq \frac{C}{\lvert\zeta\rvert}\ \mbox{and} \ \lVert \nabla_{x}R'\rVert_{L^{2}(\Omega)}\leq C
		\end{align}
		with similar $C$ as before.
	\end{remark}

 Now we use the CGO solutions. Let $a\in\mathbb{R}^{n}$ and choose unit vectors $\xi,\eta\in\mathbb{R}^{n}$  such that 
 \begin{align*}
 a\cdot\xi=a\cdot\eta=\xi\cdot\eta=0.
 \end{align*}
 Further choose $r, s>0$ so that
 \begin{align*}
 r^{2}=\frac{|a|^{2}}{4}+s^{2}.
 \end{align*}
 Define $\zeta_1,\zeta_2\in\mathbb{C}^{n}$ by 
 \begin{align}\label{definition of zeta i}
 \zeta_{1}:= r\eta+i\lb \frac{a}{2}+s\xi\rb,\ \ \zeta_{2}:=-
 r\eta+i\lb \frac{a}{2}-s\xi\rb.
 \end{align}
 Then if we denote $\rho=\eta+i\xi$, we have 
 \begin{align*}
 \begin{aligned}
 \zeta_{i}\cdot\zeta_{i}=0,\ \ \lim_{s\rightarrow \infty}\frac{\zeta_{1}}{s}=\rho,\,\, \lim_{s\rightarrow \infty}\frac{\zeta_{2}}{s}=-\rho.
 \end{aligned}
 \end{align*}
 By using the expression for the solutions $\widehat{u_{1}^{f}}(\tau,x)$ and $\widehat{u_{1}^{g}}(\tau,x)$ of the form given in \eqref{CGO to widehat u} with values of $\zeta_{i}$ given by \eqref{definition of zeta i} in \eqref{Integral identity after using useful lemma} and dividing by $s^{2}$ and taking $s\rightarrow\infty$, we have  
 \begin{equation}\label{integral identity after using CGO}
 \int\limits_{\Omega}\sum_{j=1}^{n}\sum_{k,l=1}^{n}\rho_{k}\rho_{l}c_{kl}^{j}(x)m^{2}(x)\partial_{j}\widehat{w}(\tau,x)e^{ia\cdot x}dx=0,\,\, a\in\mathbb{R}^{n}.
 \end{equation}
 
 \medskip In deriving the above identity we have used the following remark.
\begin{remark}${}$
\par
Note that the last three terms in \eqref{Integral identity after using useful lemma} contains the terms like $F_{k}(s,x)$ and $G_{k}(s,x)$ which are given by 
		\[F_{k}(s,x)=s\partial_{k}u_{1}^{g}(s,x)\ \mbox{and} \ G_{k}(s,x)= s^{2}(\partial_{k}u_{1}^{g})'(s,x).\] 
Then we have \[\widehat{F_{k}}(\tau,x)=-\widehat{\lb \PD_{k}u_{1}^{g}\rb'}(\tau,x)\,\,\mbox{and}\,\, \widehat{G_{k}}(\tau,x)=\widehat{\lb \PD_{k}\lb u_{1}^{g}\rb'\rb}''(\tau,x).\]
Therefore, substituting these expressions in the second term of \eqref{Integral identity after using useful lemma}, we have 
	\begin{align*}
	\begin{aligned}
	&4\sum_{j=1}^{n}\sum_{k,l=1}^{n}\int\limits_{\Omega}c_{kl}^{j}(x)\left\{\widehat{(\partial_{k}u_{1}^{f})'}(\tau,x)\widehat{F_{l}}(\tau,x)+\widehat{F_{k}}\widehat{(\partial_{l}u_{1}^{f})'}(\tau,x)\right\}\partial_{j}\widehat{w}(\tau,x)\D x\\
	&\ \ \ \ =-4\sum_{j=1}^{n}\sum_{k,l=1}^{n}\int\limits_{\Omega}e^{(\zeta_{1}+\zeta_{2})\cdot x}c_{kl}^{j}(x)\Big\{\widehat{(\partial_{k}u_{1}^{f})'}(\tau,x)\widehat{(\partial_{l}u_{1}^{g})}'(\tau,x)+\widehat{(\partial_{l}u_{1}^{f})'}(\tau,x)\widehat{(\partial_{k}u_{1}^{g})}'(\tau,x)\Big\}\partial_{j}\widehat{w}(\tau,x)\D x.
	\end{aligned}
	\end{align*}	
	We will give the explanation for the first part, then the second part can be explained similarly. Denote $\zeta_i=(\zeta_i^1,\cdots,\zeta_i^n),\,i=1,2$. Then we have
	\begin{align*}
	\begin{aligned}
	M_{1}&:=-4\sum_{j=1}^{n}\sum_{k,l=1}^{n}\int\limits_{\Omega}
	e^{(\zeta_{1}+\zeta_{2})\cdot x}c_{kl}^{j}(x)\widehat{(\partial_{k}u_{1}^{f})'}(\tau,x)\widehat{(\partial_{l}u_{1}^{g})}'(\tau,x)\D x \\
	&=4\sum_{j=1}^{n}\sum_{k,l=1}^{n}\int\limits_{\Omega}c_{kl}^{j}(x)e^{(\zeta_{1}+\zeta_{2})\cdot x}\Big\{\zeta_{1}^{k}\zeta_{2}^{l}m^{2}(x)R_{1}'R_{2}'+\zeta_{2}^{l}m(x)\PD_{k}m(x)R_{2}'R_{1}'+\zeta_{2}^{l}m^{2}(x)\PD_{k}R_{1}'R_{2}'\\
	&\ \ \ \ \ \ \  \ \ \ \ \ \ \ \ \ \ \ \ +\zeta_{1}^{k}m(x)\PD_{l}m(x)R_{1}'R_{2}'+\PD_{k}m(x)\PD_{l}m(x)R_{1}'R_{2}'+m(x)\PD_{l}m(x)\PD_{k}R_{1}'R_{2}'\\
	&\ \ \ \ \ \ \ \ \ \ \ \ \ \ \ \ \ \ \ +m^{2}\zeta_{1}^{k}R_{1}'\PD_{l}R_{2}'+m(x)\PD_{k}m(x)R_{1}'\PD_{l}R_{2}'+m^{2}\PD_{k}R_{1}'\PD_{l}R_{2}' \Big\}\D x.
	\end{aligned}
	\end{align*}
	Now using the values of $\zeta_{i}$ and dividing by $s^{2}$ and  taking $s\rightarrow\infty$, we get  
	\[\lim_{s\rightarrow\infty}M_{1}=0\]
	where in deriving the above identity, we have used \eqref{Estimate for tau derivative of error term in CGO}.
	If we substitute the CGO in \eqref{Integral identity after using useful lemma} and divide by $s^{2}$, a similar argument shows that the last three terms in \eqref{Integral identity after using useful lemma} go to zero as $s\rightarrow\infty$. 
	
\end{remark} 
From \eqref{integral identity after using CGO}, we have
 \begin{align}\label{final identity}
 \sum_{j=1}^{n}\Bigg[\sum_{k,l=1}^{n}\rho_{k}\rho_{l}c_{kl}^{j}(x)m^{2}(x)\Bigg]\partial_{j}\widehat{w}(x)=0, \ \ x\in\Omega.
 \end{align}
 To prove that $c_{kl}^{j}(x)=0$ in $\Omega $, we will follow the argument given in \cite{KN} and make use of the following lemma similar to Lemma 3.1 of \cite{KN} in our setup.
 \begin{lemma}\label{Linearly independent solution to conductivity equation}
 	Suppose $n\geq 2$. For each fix $\tau$, there exist solutions $v_{j}(\tau,x)\in H^{2}(\Omega) $, $1\leq j\leq n$ of 
 	\begin{align*}
 	\tau^{2}v-\nabla_{x}\cdot(\gamma(x)\nabla_{x}v)=0,\ x\in\Omega
 	\end{align*}
 	such that 
 	\begin{align*}
 	{\rm det}\Big(\frac{\partial v_{j}}{\partial x_{i}}\Big)\neq 0, \ \text{a.e.} \ x\in\Omega.
 	\end{align*}
 	\begin{proof}
 		Consider the set $V$ defined by 
 		\begin{align}\label{definition of V}
 		V:=\{\rho\in\mathbb{C}^{n}: \rho\cdot\rho=0,\ |\rho|=\sqrt{2}\}. 
 		\end{align}
 		Let $\rho^{1},\rho^{2},\cdots,\rho^{n}\in V$ be n-linearly independent vectors over $\mathbb{C}$. Now let 
 		\begin{align}\label{definition of vj}
 		v_{j}(\tau,x):= m(x)e^{r\rho^{j}\cdot x}(1+R_{j}(x)), \ x\in\Omega,\, r>0, \ j=1,2, \cdots,n
 		\end{align}
 		as before. Then, we have 
 		\begin{align*}
 		\begin{aligned}
 		\text{det}
 		\begin{bmatrix}
 		\nabla_{x}v_{1}\\
 		\nabla_{x}v_{2}\\
 		\nabla_{x}v_{3}\\
 		\cdot\\
 		\cdot\\
 		\cdot\\
 		\nabla_{x}v_{n}
 		\end{bmatrix}
 		=r^{n}m^{n}(x)e^{r(\rho^{1}+\rho^{2}+\cdots+\rho^{n})\cdot x }\text{det}
 		\begin{bmatrix}
 		\rho^{1}(1+R_{1})+O(r^{-1})\\
 		\rho^{2}(1+R_{2})+O(r^{-1})\\
 		\rho^{3}(1+R_{3})+O(r^{-1})\\
 		\cdot\\
 		\cdot\\
 		\cdot\\
 		\rho^{n}(1+R_{n})+O(r^{-1})
 		\end{bmatrix}
 		\end{aligned}
 		\end{align*}
 		Now from \eqref{Hs estimate on R} using the Sobolev embedding, we have $||R_{j}||_{L^{\infty}(\Omega)}\leq \frac{C}{r}$ , we have 
 		\begin{align*}
 		\text{det}
 		\begin{bmatrix}
 		\nabla_{x}v_{1}\\
 		\nabla_{x}v_{2}\\
 		\nabla_{x}v_{3}\\
 		\cdot\\
 		\cdot\\
 		\cdot\\
 		\nabla_{x}v_{n}
 		\end{bmatrix}
 		\neq 0\,\,  \text{for sufficiently large}\ r.
 		\end{align*}
 	\end{proof}
 \end{lemma}
 Using lemma \ref{Linearly independent solution to conductivity equation}, there exist solutions $\{\widehat{w}_{j}\}_{1\leq j\leq n}$ of \eqref{equation for widehat u} in such way that $\{\nabla_{x}\widehat{w}_{j}\}_{1\leq j\leq n}$  are linearily independent.  Taking these choices of $\{\widehat{w}_{j}\}_{1\leq j\leq n}$ in \eqref{final identity}, we have
 \begin{align}\label{vanishing of rhok rhol cjkl}
 \sum_{k,l=1}^{n}\rho_{k}\rho_{l}c_{kl}^{j}(x)=0,\ x\in\Omega,\,\, j=1,2,\cdots,n.
 \end{align}
 Now since \eqref{vanishing of rhok rhol cjkl} holds for all $\rho\in V$, where $V$ is defined as in \eqref{definition of V}. Now using $\rho=ze_{k}+\sqrt{-1}z e_{l}$ for $z\in {\mathbb R}$ with $|z|=1$ and $k\neq l$, where $e_{j}'s$ for $1\leq j\leq n$ are standard basis for $\mathbb{R}^{n}$, in \eqref{vanishing of rhok rhol cjkl}, we have 
 \begin{align*}
 z^{2}c_{kk}^{j}(x)-z^{2}c_{ll}^{j}(x)+iz^{2}c_{kl}^{j}(x)=0,\,\, x\in\Omega,\,\,k\neq l.
 \end{align*}
 From here, we will have 
 \begin{align}\label{vanishing og cjkl and eequality of cjkk and cjll  for k and l are not same}
 c_{kk}^{j}(x)=c_{ll}^{j}(x),\ c_{kl}^{j}(x)=0 \ \text{for all $j$ and $k\neq l$}.
 \end{align}
 Now using \eqref{vanishing og cjkl and eequality of cjkk and cjll  for k and l are not same} in \eqref{integral identity}, we have 
 \begin{align}\label{Integral identity for cjkk}
 \sum_{j=1}^{n}\sum_{k=1}^{n}\int\limits_{\Omega}c_{kk}^{j}(x)\partial_{k}\widehat{u_{1}^{f}}(\tau,x)\partial_{k}\widehat{u_{1}^{g}}(\tau,x)\partial_{j}\widehat{w}(\tau,x)dx=0
 \end{align}
 holds for all $\widehat{u_{1}^{f}},\widehat{u_{1}^{g}}$ and $\widehat{w}$ satisfying \eqref{equation for widehat u}. Now using CGO solutions for $\widehat{u_{1}^{g}}$ and  $\widehat{w}$, we have 
 \begin{align}\label{prefinal equation}
 \sum_{j=1}^{n}\sum_{k=1}^{n}\rho_{j}\rho_{k}c^{j}_{kk}(x)\partial_{k}\widehat{u_{1}^{f}}(\tau,x)=0.
 \end{align}
 Note that \eqref{prefinal equation} holds for any $f\in C_0^\infty({\mathbb R}\times \partial\Omega)$. Hence by taking $n$ different $f$ and using Lemma \ref{Linearly independent solution to conductivity equation}, we have
 \begin{align*}
 \sum_{j=1}^{n}\sum_{k=1}^{n}\rho_{j}\rho_{k}c^{j}_{kk}(x)=0.
 \end{align*}
 Now take same $\rho\in V$ as defined before, we have 
 \begin{align*}
 c^{j}_{kk}(x)=0\,\,\text{for all} \ 1\leq j,k\leq n \ \text{and} \ x\in\Omega.
 \end{align*}
 Hence combining this and \eqref{vanishing og cjkl and eequality of cjkk and cjll  for k and l are not same}, we have 
 \begin{align*}
 c_{kl}^{j}(x)=0\,\, \text{in}\ \Omega\,\, \text{for all}\,\,1\leq j,k,l\leq n.
 \end{align*}
 This completes the proof.

 \bigskip
 \noindent
 {\bf \large Acknowledgement}${}$
 \par
 First of all we thank the anonymous referees' 
 useful comments which helped us to improve this paper. We also thank the several research funds which supported this study and they are as follows. The first author was partially supported by Grant-in-Aid for Scientific Research (15K21766) of the Japan Society for the Promotion of Science doing the research of this paper. The second author would like to thank his advisor Venky Krishnan for stimulating discussions. The second author was partially supported by Grant-in-Aid for Scientific Research (15H05740) of the Japan Society for the Promotion of Science doing the research of this paper. He also benefited from the support of Airbus Group Corporate Foundation Chair ``Mathematics of Complex Systems'' established at TIFR Centre for Applicable Mathematics  and  TIFR  International  Centre  for  Theoretical  Sciences,  Bangalore, India. 
 
 \bigskip
 \noindent
\appendix
 \section{}\label{subsection A}
 \setcounter{equation}{0}
 \renewcommand{\theequation}{A.\arabic{equation}}
 
 We will give here some argument for justifying the $\epsilon$-expansion which also gives some argument for the unique solvability of \eqref{equation of interest}.
We look for a solution $u(t,x)$ to \eqref{equation of interest} of the form
 \begin{equation}\label{form of u}
 \begin{aligned}
 u(t,x):=\epsilon\left\{u_{1}(t,x)+\epsilon\left(u_{2}(t,x)+w(t,x)\right)\right\},
 \end{aligned}
 \end{equation}
 where $u_1$, $u_2$ are the solutions to the intial boundary value problems \eqref{equation for u1 i} and \eqref{equation for u2 i}, respectively.
 Then, $w(t,x)$ has to satisfy
 \begin{align*}
 \begin{aligned}
 \begin{cases}
 &\partial_{t}^{2}w(t,x)-\nabla_{x}\left(\gamma(x)\nabla_{x}w(t,x)\right)=\epsilon^{-2}\nabla_{x}\cdot R(x,\epsilon\nabla_{x}u_{1}+\epsilon^{2}\nabla_{x}u_{2}+\epsilon^{2}\nabla_{x}w)\\
 &+\epsilon\sum_{j=1}^{n}\partial_{j}\left(\sum_{k,l=1}^{n}c_{kl}^{j}(x)\left(\partial_{k}u_{1}\partial_{l}u_{2}+\partial_{l}u_{1}\partial_{k}u_{2}\right)(t,x)\right)\\
 &+\epsilon\sum_{j=1}^{n}\partial_{j}\left(\sum_{k,l=1}^{n}c_{kl}^{j}(x)\left(\partial_{k}u_{1}\partial_{l}w+\partial_{l}u_{1}\partial_{k}w\right)(t,x)\right)\\
 &+\epsilon^{2}\sum_{j=1}^{n}\partial_{j}\left(\sum_{k,l=1}^{n}c_{kl}^{j}(x)\left(\partial_{k}u_{2}\partial_{l}u_{2}+\partial_{k}u_{2}\partial_{l}w+\partial_{l}u_{2}\partial_{k}w+\partial_{k}w\partial_{l}w\right)(t,x)\right)\\
 &\qquad\qquad\qquad\qquad\qquad\qquad\qquad\qquad\qquad\qquad\qquad\qquad\qquad\qquad\qquad\,\, \ \text{in}\  Q_{T},\\
 &w(0,x)=\partial_{t}w(0,x)=0,\ x\in\Omega \ \text{and}\ w|_{\partial Q_{T}}=0.
 \end{cases}
 \end{aligned}
 \end{align*}
 By the mean value theorem, we have
 \begin{align*}
 R(x,\epsilon\nabla_{x}u_{1}+\epsilon^{2}\nabla_{x}u_{2}+\epsilon^{2}\nabla_{x}w)&=R(x,\epsilon\nabla_{x}u_{1}+\epsilon^{2}\nabla_{x}u_{2})\\
 &+\int\limits_{0}^{1}\frac{d}{d\theta}R(x,\epsilon\nabla_{x}u_{1}+\epsilon^{2}\nabla_{x}u_{2}+\theta\epsilon^{2}\nabla_{x}w)d\theta\\
 &=R(x,\epsilon\nabla_{x}u_{1}+\epsilon^{2}\nabla_{x}u_{2})+\epsilon^{3} K(x,\epsilon\nabla_{x}w;\epsilon)\nabla_{x}w
 \end{align*}
 where 
 \begin{align*}
 \epsilon K(x,\epsilon\nabla_{x}w;\epsilon):=\int\limits_{0}^{1}D_{q}R(x,\epsilon\nabla_{x}u_{1}+\epsilon^{2}\nabla_{x}u_{2}+\theta\epsilon^{2}\nabla_{x}w)d\theta
 \end{align*}
 where $D_{q}R(x,q)=\left(\left(\partial_{q_{i}}R_{j}\right)\right)_{1\leq i,j\leq n}$ and $K_{ij}=\partial_{q_{i}}R_{j}$. After using this in previous equation, we get
 \begin{align*}
 \begin{aligned}
 \begin{cases}
 &\partial_{t}^{2}w(t,x)-\nabla_{x}\left(\gamma(x)\nabla_{x}w(t,x)\right) =\epsilon\sum_{j=1}^{n}\partial_{j}\left(\sum_{k,l=1}^{n}c_{kl}^{j}(x)\left(\partial_{k}u_{1}\partial_{l}u_{2}
 +\partial_{l}u_{1}\partial_{k}u_{2}\right)\right)\\
 &+\epsilon\sum_{j=1}^{n}\sum_{k,l=1}^{n}\partial_{j}\left(c_{kl}^{j}(x)\partial_{k}u_{1}\right)\partial_{l}w
 +\epsilon\sum_{j=1}^{n}\sum_{k,l=1}^{n}\left(c_{kl}^{j}(x)\partial_{k}u_{1}\right)\partial_{jl}^{2}w\\
 &+\epsilon\sum_{j=1}^{n}\sum_{k,l=1}^{n}\partial_{j}\left(c_{kl}^{j}(x)\partial_{l}u_{1}\right)\partial_{k}w+\epsilon\sum_{j=1}^{n}\sum_{k,l=1}^{n}\left(c_{kl}^{j}(x)\partial_{k}u_{1}\right)\partial_{jk}^{2}w\\
 &+\epsilon^{2}\sum_{j=1}^{n}\partial_{j}\left(\sum_{k,l=1}^{n}c_{kl}^{j}(x)\partial_{k}u_{2}\partial_{l}u_{2}\right)+\epsilon^{2}\sum_{j=1}^{n}\sum_{k,l=1}^{n}\partial_{j}\left(c_{kl}^{j}(x)\partial_{k}u_{2}\right)\partial_{l}w\\
 &+\epsilon^{2}\sum_{j=1}^{n}\sum_{k,l=1}^{n}\left(c_{kl}^{j}(x)\partial_{k}u_{2}\right)\partial_{jl}^{2}w
 +\epsilon^{2}\sum_{j=1}^{n}\sum_{k,l=1}^{n}\partial_{j}\left(c_{kl}^{j}(x)\partial_{l}u_{2}\right)\partial_{k}w\\
 &+\epsilon^{2}\sum_{j=1}^{n}\sum_{k,l=1}^{n}\left(c_{kl}^{j}(x)\partial_{k}u_{2}\right)\partial_{jk}^{2}w
 +\epsilon^{2}\sum_{j=1}^{n}\sum_{k,l=1}^{n}c_{kl}^{j}(x)\left(\partial_{jk}^{2}w\partial_{l}w+\partial_{k}w\partial_{jl}^{2}w\right)\\
 &+\epsilon^{2}\sum_{j=1}^{n}\sum_{k,l=1}^{n}\partial_{j}c_{kl}^{j}(x)\partial_{k}w\partial_{l}w+\epsilon\sum_{i=1}^{n}\sum_{j=1}^{n}\partial_{i}K_{ij}(x,\epsilon\nabla_{x}w;\epsilon)\partial_{j}w\\
 &+\epsilon\sum_{i=1}^{n}\sum_{j=1}^{n}K_{ij}(x,\epsilon\nabla_{x}w;\epsilon)\partial^{2}_{ij}w+\epsilon^{-2}\nabla_{x}\cdot R(x,\epsilon\nabla_{x}u_{1}+\epsilon\nabla_{x}u_{2})\\
 &\qquad\qquad\qquad\qquad\qquad\qquad\qquad\qquad\qquad\qquad\qquad\qquad\qquad\qquad\qquad\qquad\,\, \ \text{in}\  Q_{T},\\
 &w(0,x)=\partial_{t}w(0,x)=0,\ x\in\Omega \ \text{and}\ w|_{\partial Q_{T}}=0.
 \end{cases}
 \end{aligned}
 \end{align*}
 Thus, finally we have derived the following initial boundary value problem for $w$:
 \begin{align}\label{equation for w}
 \begin{aligned}
 \begin{cases}
 &\partial_{t}^{2}w(t,x)-\nabla_{x}\cdot\left(\gamma(x)\nabla_{x}w(t,x)\right)=\epsilon\sum_{j=1}^{n}\partial_{j}\left(\sum_{k,l=1}^{n}c_{kl}^{j}(x)\left(\partial_{k}u_{1}\partial_{l}u_{2}+\partial_{l}u_{1}\partial_{k}u_{2}\right)(t,x)\right)\\
 &+\epsilon\sum_{j=1}^{n}\sum_{k,l=1}^{n}\partial_{j}\left(c_{kl}^{j}(x)\partial_{k}u_{1}\right)\partial_{l}w
 +\epsilon\sum_{j=1}^{n}\sum_{k,l=1}^{n}\left(c_{kl}^{j}(x)\partial_{k}u_{1}\right)\partial_{jl}^{2}w\\
 &+\epsilon\sum_{j=1}^{n}\sum_{k,l=1}^{n}\partial_{j}\left(c_{kl}^{j}(x)\partial_{l}u_{1}\right)\partial_{k}w
 +\epsilon\sum_{j=1}^{n}\sum_{k,l=1}^{n}\left(c_{kl}^{j}(x)\partial_{l}u_{1}\right)\partial_{jk}^{2}w\\
 &+\epsilon^{2}\sum_{j=1}^{n}\partial_{j}\left(\sum_{k,l=1}^{n}c_{kl}^{j}(x)\partial_{k}u_{2}\partial_{l}u_{2}\right)+\epsilon^{2}\sum_{j=1}^{n}\sum_{k,l=1}^{n}\partial_{j}\left(c_{kl}^{j}(x)\partial_{k}u_{2}\right)\partial_{l}w\\
 &+\epsilon^{2}\sum_{j=1}^{n}\sum_{k,l=1}^{n}\left(c_{kl}^{j}(x)\partial_{k}u_{2}\right)\partial_{jl}^{2}w
 +\epsilon^{2}\sum_{j=1}^{n}\sum_{k,l=1}^{n}\partial_{j}\left(c_{kl}^{j}(x)\partial_{l}u_{2}\right)\partial_{k}w\\
 &+\epsilon^{2}\sum_{j=1}^{n}\sum_{k,l=1}^{n}\left(c_{kl}^{j}(x)\partial_{k}u_{2}\right)\partial_{jk}^{2}w
 +\epsilon^{2}\sum_{j=1}^{n}\sum_{k,l=1}^{n}c_{kl}^{j}(x)\left(\partial_{jk}^{2}w\partial_{l}w+\partial_{k}w\partial_{jl}^{2}w\right)\\
 &+\epsilon^{2}\sum_{j=1}^{n}\sum_{k,l=1}^{n}\partial_{j}c_{kl}^{j}(x)\partial_{k}w\partial_{l}w+\epsilon^{-2}\nabla_{x}\cdot R(x,\epsilon\nabla_{x}u_{1}+\epsilon\nabla_{x}u_{2})\\
 &+\epsilon\sum_{i=1}^{n}\sum_{j=1}^{n}\partial_{i}K_{ij}(x,\epsilon\nabla_{x}w;\epsilon)\partial_{j}w+\epsilon\sum_{i=1}^{n}\sum_{j=1}^{n}K_{ij}(x,\epsilon\nabla_{x}w;\epsilon)\partial^{2}_{ij}w\\
 &\qquad\qquad\qquad\qquad\qquad\qquad\qquad\qquad\qquad\qquad\qquad\qquad\qquad\qquad\qquad\qquad\qquad\,\,\,\, \ \text{in}\  Q_{T},\\
 &w(0,x)=\partial_{t}w(0,x)=0,\ x\in\Omega \ \text{and}\ w|_{\partial Q_{T}}=0.
 \end{cases}
 \end{aligned}
 \end{align}
 In order to simplify the description of the above equation for $w$, let us introduce the following notations: 
 \begin{align*}
 \begin{aligned}
 \begin{cases}
 &A(w(t))w=\nabla_{x}\cdot\left(\gamma(x)\nabla_{x}w(t,x)\right)+\epsilon\Gamma(x,\nabla_{x}w;\epsilon)\cdot \partial_x^{2}w,\\
 &\epsilon\Gamma(x,\nabla_{x}w;\epsilon):=\epsilon\left(\sum_{k=1}^{n}c_{kl}^{j}(x)\partial_{k}u_{1}\right)_{1\leq j,l\leq n}+\epsilon\left(\sum_{l=1}^{n}c_{kl}^{j}(x)\partial_{l}u_{1}\right)_{1\leq j,k\leq n}\\
 &+\epsilon^{2}\left(\sum_{k=1}^{n}c_{kl}^{j}(x)\partial_{k}u_{2}\right)_{1\leq j,l\leq n}
 +\epsilon^{2}\left(\sum_{l=1}^{n}c_{kl}^{j}(x)\partial_{l}u_{2}\right)_{1\leq j,k\leq n}\\
 &+\epsilon^{2}\left(\sum_{k=1}^{n}c_{kl}^{j}(x)\partial_{k}w\right)_{1\leq j,l\leq n}+\epsilon^{2}\left(\sum_{l=1}^{n}c_{kl}^{j}(x)\partial_{l}w\right)_{1\leq j,k\leq n}+\Bigg(K_{ij}(x,\epsilon\nabla_{x}w;\epsilon)\Bigg)_{1\leq i,j\leq n},\\
 &\epsilon\vec{G}(x,\nabla_{x}w;\epsilon):=\epsilon\left(\sum_{j=1}^{n}\sum_{k=1}^{n}\partial_{j}\left(c_{kl}^{j}(x)\partial_{k}u_{1}\right)\right)_{1\leq l\leq n}\\
 &+\epsilon\left(\sum_{j=1}^{n}\sum_{l=1}^{n}\partial_{j}\left(c_{kl}^{j}(x)\partial_{k}u_{1}\right)\right)_{1\leq k\leq n} +\epsilon^{2}\left(\sum_{j=1}^{n}\sum_{k=1}^{n}\partial_{j}\left(c_{kl}^{j}(x)\partial_{k}u_{2}\right)\right)_{1\leq l\leq n}\\
 &+\epsilon^{2}\left(\sum_{j=1}^{n}\sum_{l=1}^{n}\partial_{j}\left(c_{kl}^{j}(x)\partial_{k}u_{2}\right)\right)_{1\leq k\leq n}+\epsilon^{2}\left(\sum_{j=1}^{n}\sum_{k=1}^{n}\partial_{j}c_{kl}^{j}(x)\partial_{k}w\right)_{1\leq l\leq n}\\
 &+\epsilon\Bigg(\sum_{i=1}^{n}\partial_{i}K_{ij}(x,\epsilon\nabla_{x}w;\epsilon)\Bigg)_{1\leq j\leq n},\\
 &\epsilon F(x,\nabla_{x},\nabla_{x}u_{2};\epsilon):=\epsilon\sum_{j=1}^{n}\partial_{j}\left(\sum_{k,l=1}^{n}c_{kl}^{j}(x)\left(\partial_{k}u_{1}\partial_{l}u_{2}+\partial_{l}u_{1}\partial_{k}u_{2}\right)\right)\\
 &+\epsilon^{2}\sum_{j=1}^{n}\partial_{j}\left(\sum_{k,l=1}^{n}c_{kl}^{j}(x)\partial_{k}u_{2}\partial_{l}u_{2}\right)+\epsilon^{-2}\nabla_{x}\cdot R(x,\epsilon\nabla_{x}u_{1}+\epsilon^{2}\nabla_{x}u_{2}).
 \end{cases}
 \end{aligned}
 \end{align*}
 Here $\partial_x^2 w=(\partial_{ij} w)_{1\le i,j\le n}$ and $\Gamma(x,\nabla_{x}w;\epsilon)\cdot \partial_x^{2}w$ denotes the real inner product of $\Gamma(x,\nabla_{x}w;\epsilon)$ and $\partial_x^{2}w$. Using the above notations, \eqref{equation for w} becomes 
 \begin{align}\label{equation for w in final form}
 \begin{aligned}
 \begin{cases}
 &\partial_{t}^{2}w-A(w(t))w-\epsilon \vec{G}(x,\epsilon\nabla_{x}w;\epsilon)\cdot\nabla_{x}w=\epsilon F(x,\nabla_{x},\nabla_{x}u_{2};\epsilon)\,\,\text{in $Q_T$},\\
 &w(0,x)=\partial_{t}w(0,x)=0,\ x\in\Omega \ \text{and}\ w|_{\partial Q_{T}}=0.
 \end{cases}
 \end{aligned}
 \end{align}
 
 The justification of $\epsilon$-expansion is given as the following theorem.
 \begin{theorem}\label{unique solvability for w}
 	Let $m\ge [n/2]+3$ and $f\in B_M$. Then, there exists $\epsilon_0=\epsilon_0(h,T, m, M)>0$ and $w=w(t,x;\epsilon)\in X_m$ for $0<\epsilon<\epsilon_0$ such that each $w=w(\cdot,\cdot;\epsilon)$ is the unique solution to the initial boundary value problem
 	\eqref{equation for w in final form} with the estimate
 	\begin{equation}\label{estimate of w}
 	|||w|||_m:=\left\{\sum_{j=0}^m\sum_{k=0}^{m-j} \displaystyle\sup_{t\in[0,T]}\Vert\partial_t^k w(t,\cdot;\epsilon)\Vert_{H^{m-j-k}(\Omega)}^2\right\}^{1/2}=O(\epsilon)\,\,\text{\rm as $\epsilon\rightarrow 0$}.
 	\end{equation}
 	Here $h$ and $B_M$, $X_m$ were defined in Section 1 right after \eqref{definition of P(x,q)} and 
 	the paragraph after \eqref{estimate of R}, respectively.
 \end{theorem}
 
 Note that once we have Theorem \ref{unique solvability for w}, we also have the unique solvability for the initial boundary value problem \eqref{equation of interest}. The proof of Theorem \ref{unique solvability for w} can be given along the same line as the proof for the case $n=1$ which is
 given in \cite{NW}. In order to apply the same proof given in \cite{NW} for our case,
 we used the same notations as in \cite{NW} for \eqref{equation for w in final form}.
 \section{}
 \setcounter{equation}{0}
 \renewcommand{\theequation}{B.\arabic{equation}}
 
 We will give here some asymptotic property of solution $\widehat{u_{1}^{f}}(\tau,x)$. We only need this to say that the differentiation of this solution with respect to each $x_j$ gives an effect to its asymptotic property up to multiplication by $O(|\tau|)$. We refer this as {\it our task}. Although there could be a more simpler way to show this, we chose as a tool to show this the theory of pseudodifferential operators with large parameter which is equivalent to the semiclassical analysis. We haven't seen any complete study of semiclassical analysis for the Poisson operator for $\tau^2\bullet-\nabla\cdot(\gamma\nabla\bullet)$
 which could immediately give what we need. Clearly it is enough to estimate the solution $v=v(\tau)$ of
 \begin{equation}\label{equation of v}
 \left\{
 \begin{array}{ll}
 \tau^2 v-\nabla\cdot(\gamma\nabla v)=0\,&\,\text{\rm in}\,\,\Omega,\\
 v=\phi\in C^\infty(\partial\Omega)\,&\,\text{\rm on}\,\,\partial\Omega,
 \end{array}
 \right.
 \end{equation}
 By a partition of unity and the well-posedness of \eqref{equation of v}, we will see
 it is enough to analyze the asymptotic behavior of $v(\tau)$ near $\partial\Omega$. Let $y=(y',y_n)=(y_1,\cdots, y_{n-1}, y_n)$ be the boundary normal coordinates near $\partial\Omega$ such that $\partial\Omega$ and $\Omega$ are locally given as $\partial\Omega=\{y_n=0\}$ and
 $\Omega=\{y_n>0\}$, respectively. Then \eqref{equation of v} changes to
 \begin{equation}
 \left\{
 \begin{array}{ll}
 Pv=0\,&\,\text{\rm in $y_n>0$},\\
 v=\phi\,&\,\text{\rm on $y_n=0$},
 \end{array}
 \right.
 \end{equation}
 where $P=P_2+P_1$ with
 \begin{equation}
 \left\{
 \begin{array}{ll}
 P_2=D_{y_n}^2+\sum_{i,j=1}^{n-1} g^{ij} D_i D_j+\gamma^{-1}\tau^2,\\
 P_1=-i\gamma^{-1}\big(\partial_n\gamma D_n+\sum_{i,j=1}^{n-1} \partial_i(\gamma g^{ij})D_j\big),\\
 g=(g^{ij})=(\nabla_x y)(\nabla_x y)^\text{\it t},\,\,g^{nj}=0\,\,(j=1,\cdots, n-1),\,1\,\,(j=n),\\
 \partial_i=\partial_{y_i},\,\,D_j=-\sqrt{-1}\partial_j\,\,(j=1,\cdots, n),\\
 \text{\it t}\,\, \text{\rm denotes the transpose of matrices}
 \end{array}
 \right.
 \end{equation}
 and we have abused the original notations $v$ and $\phi$ to denote the corresponding ones in terms of the boundary normal coordinates. 
 
 We will look for $v$ in the form
 \begin{equation}
 v=v(y,\tau)=(\text{\rm Op}(a)\phi)(y)=(2\pi)^{-n+1}\int_{\mathbb{R}^2}
 e^{\sqrt{-1}y'\cdot\eta'}a(y,\eta',\tau)\widehat{\phi}(\eta')\,d\eta',
 \end{equation}
 where $y=(y',y_n)$, $\eta'=(\eta_1,\cdots,\eta_{n-1})$ and $\widehat\phi(\eta')=\int_{\mathbb{R}^{n-1}}
 e^{-\sqrt{-1}y'\cdot\eta'}\phi(y')\,dy'$. Here $\text{\rm Op}(a)$ is a pseudodifferential operator with a large parameter $\tau$ depending smoothly on $y_n\ge0$. More precisely, the symbol $a=a(y,\eta',\tau)\in S(0)$ of $\text{\rm Op}(a)$ is defined as follows. 
 \par
 For $m\in\mathbb{R}$, $a(y,\eta',\tau)\in S(m)$ if and only if it satisfies the following conditions (i) and (ii).
 \begin{itemize}
 	\item[(i)]
 	\begin{equation*}\,\, a(y,\eta',\tau)\in C^\infty(\mathbb{R}_{y'}^{n-1}\times\overline{\mathbb{R}_{y_n}^+}\times\mathbb{R}_{\eta'}^{n-1}\times\{\tau\ge1\}),\\
 	\end{equation*}
 	\item[(ii)]
 	\begin{equation*}
 	\begin{array}{ll}\,\,\text{For any}\,\alpha,\,\beta\in\mathbb{Z}_+\,\text{with}\,\mathbb{Z}_+:=\mathbb{N}\cup\{0\},\,\,\text{there exists a constant}\,\,
 	C_{\alpha,\beta,k}>0\,\,\text{such that}\\
 	|D_{y'}^\alpha D_{\eta'}^\beta a(y,\eta',\tau)|\le C_{\alpha,\beta,k} \langle\eta';\tau\rangle^{m+k-|\beta|}\\
 	\qquad\text{for}\,\,(y',y_n,\eta',\tau)\in \mathbb{R}_{y'}^{n-1}\times\overline{\mathbb{R}_{y_n}^+}\times\mathbb{R}_{\eta'}^{n-1}\times\{\tau\ge1\},
 	\end{array}
 	\end{equation*}
 	where $\langle\eta';\tau\rangle^2=1+|\eta'|^2+\tau^2$, $\mathbb{R}_{y'}^{n-1}=\{y'\in\mathbb{R}^{n-1}\}$,
 	$\mathbb{R}_{y_n}^+=\{y_n\in\mathbb{R}: y_n>0\}$.
 \end{itemize}
 
 \noindent
 Denote $S[m]=\{\text{Op}(a): a\in S(m)\}$, and call $m$ the orders of $a$ and $\text{Op}(a)$.
 
 Consider $P=P(y,D_y,\tau)$ not only as a differential operators with respect to $y_3$, but also consider it as a pseudodifferential operator with respect to $y'$ with large parameter $\tau$. By the composition formula of pseudodifferential operators,
 \begin{equation}\label{composition}
 P(\text{Op}(a)\phi)(y)=(2\pi)^{-n+1}\int_{{\mathbb R}^2} e^{\sqrt{-1}y'\cdot\eta'}\mathcal{A}(y,\eta',D_n),\tau)\widehat\phi(\eta')\,d\eta'
 \end{equation}
 with
 \begin{equation}\label{composition symbol}
 \begin{array}{ll}
 \mathcal{A}(y,\eta',\tau)=\sum_{\alpha\in{\mathbb{Z}_+^{n-1}}}(\alpha!)^{-1}\partial_{\eta'}^\alpha p(y,\eta',D_3,\tau)D_{y'}^\alpha a(y,\eta',\tau),\\
 \\
 a\sim\sum_{\ell=1}^\infty a_{-\ell}\,\,\text{with each}\,\, a_{-\ell}\in S(-\ell),
 \end{array}
 \end{equation}
 where 
 $$\begin{array}{ll}
 \partial_{\eta'}=(\partial_{\eta_1},\cdots,\partial_{\eta_{n-1}}),\,\,D_y'=(D_{y_1},\dots,D_{y_{n-1}}),\\ p(y,\eta',D_n,\tau)=P(y,D_{y'}, D_n,\tau)\big|_{D_{y'}=\eta'}.
 \end{array}
 $$
 
 Now expand the coefficients of $p(y,\eta',D_n,\tau)$ into their Taylor series around $y_n=0$ and introduce the following concept of order.
 
 \medskip\noindent
 \begin{definition}
  We regard the multiplications by $\eta_1,\cdots,\eta_{n-1},\,\tau$ as operators of order 1, the multiplication by $y_n$ as a operator of order $-1$ and $D_{y'}$ as an operator of order $0$.
 \end{definition}
 
 \medskip\noindent
 Then, considering this new concept of order equivalently to the order of symbols, we have the following asymptotic expansion of $\mathcal{A}$ in descending order
 \begin{equation}
 \mathcal{A}\sim\sum_{\ell=0}^\infty\mathcal{A}_{-\ell}\,\,\text{with order of each $\mathcal{A}_{-\ell}=-\ell$}.
 \end{equation}
 For example $\mathcal{A}_0$ and $\mathcal{A}_{-1}$ are given as 
 \begin{equation}
 \begin{array}{ll}
 \mathcal{A}_0=P_{2,0}^{(0)} a_0,\\
 \mathcal{A}_{-1}=p_{2,0}^{(0)} a_{-1}+\sum_{j+|\alpha|=1}y_n^j P_{2,j}^{(\alpha)}D_{y'}^\alpha a_0+P_{1,0}^{(0)} a_0,
 \end{array}
 \end{equation}
 where likewise $p(y,\eta',D_n,\tau)$ we have denoted
 each $p_j(y,\eta',D_n,\tau),\,\,j=1,2$ by $$p_j(y,\eta',D_n,\tau)=P_k(y,D_{y'},D_n,\tau)\big|_{D_{y'}=\eta'},
 $$ and we have also introduced the notation $p_{j,k}^{(\alpha)}=\partial_{y_n}^k\partial_{\eta'}^\alpha p_j(y',0,\eta',D_n,\tau)$. To have $P(\text{Op}(a)\phi)(y)=0$, we set
 the conditions 
 \begin{equation}\label{eq mathcal A}
 \mathcal{A}_{-\ell}=0,\,\,\ell\in\mathbb{Z}_+.
 \end{equation}
 Also, to have $v\big|_{y_n=0}=\phi$, we impose $a_{-\ell}$'s to satisfy
 \begin{equation}\label{bc for a}
 a_{-\ell}\big|_{y_n=0}=1\,\,(\ell=0),\,\,0\,\,(\ell\ge 1).
 \end{equation}
 By further imposing that each $a_{-\ell}$ satisfies $a_\ell\rightarrow 0\,\,(y_n\rightarrow\infty)$, we can uniquely solve the equations of system \eqref{eq mathcal A}, \eqref{bc for a}. For example, by $p_{2,0}^{(0)}=D_n^2+\lambda^2(y',\eta',\tau)$ with $\lambda=\lambda(y',\eta',\tau)$ given by $$\lambda=\sqrt{\sum_{i,j=1}^{n-1} g^{ij}(y',0)\eta_i\eta_j+\gamma^{-1}(y',0)\tau^2},
 $$ we have
 \begin{equation}
 \begin{array}{ll}
 a_0=e^{-\lambda y_n}\in S(0),\\
 a_{-1}=\big(\sum_{j=1}^{n-1} f_j(y',\eta',\tau) y_n^j\big) e^{-\lambda y_n}
 \end{array}
 \end{equation}
 with each $f_j=f_j(y',\eta',\tau)\in S(j-1)$ given by 
 \begin{equation}
 \left\{\begin{array}{ll}
 f_1=(2\lambda)^{-1}e_0-(4\lambda^2)^{-1}e_1,\\
 f_2=-(4\lambda)^{-1}e_1,
 \end{array}
 \right.
 \end{equation}
 where
\begin{equation}
 \left\{
 \begin{array}{ll}
 e_0=\sqrt{-1}[\gamma^{-1}\{\lambda\partial_n\gamma+\sum_{i,j=1}^{n-1}\partial_i(\gamma g^{ij})\eta_j\}]\big|_{y_n=0},\\
 e_1=[-\{\sum_{i,j=1}^{n-1}(\partial_n g^{ij})\eta_i \eta_j+(\partial_n\gamma^{-1})\tau^2\}+2\sum_{i=1}^{n-1}\{D_i\lambda\big(\sum_{j=1}^{n-1}g^{ij}\eta_j\big)\}]
 \big|_{y_3=0}.
 \end{array}
 \right.
 \end{equation}
 
 Observe that $\langle\eta';\tau\rangle\le \langle\eta'\rangle\tau\,\,(\tau\ge2)$ with $\langle \eta'\rangle=\sqrt{|\eta'|^2+1}$\,, $\langle\eta';\tau\rangle^{-1}\le \tau^{-1}$. Then, since each $a_{-\ell}\in S(-\ell)$ and we can assume $\phi\in C_0^\infty({\mathbb{R}}^{n-1})$ due to the fact that $\phi$ is supported in a local coordinates neighborhood of $\partial\Omega$, we have the following asymptotic expansion of $v(y,\tau)$ in
 descending estimate with respect to $\tau$
 \begin{equation}
 v(y,\tau)\sim\sum_{\ell=0}^\infty v_{-\ell}(y,\tau)   
 \end{equation}
 and each $v_{-\ell}(y,\tau)$ satisfies the estimate: for any $k\in\mathbb{Z}_+$, $\alpha\in\mathbb{Z}_+^{n-1}$, there exists a constant $C_{k,\alpha}>0$ such that
 \begin{equation}
 \big|D_{y_n}^k D_{y'}^\alpha v_{-\ell}(y,\tau)\big|\le C_{k,\alpha}\tau^{-\ell+k},\,\,\,(y',y_n)\in \mathbb{R}_{y'}^{n-1}\times\overline{\mathbb{R}_{y_n}^+}.   
 \end{equation}
 Further, for each $N\in\mathbb{Z}_+$, $v^{(N)}(y,\tau)=\sum_{\ell=0}^N v_{-\ell}(y,\tau)$ satisfies
 \begin{equation}
 \left\{
 \begin{array}{ll}
 Pv^{(N)}=r^{(N)}\in C^\infty(\mathbb{R}_{y'}^{n-1}\times\overline{\mathbb{R}_{y_n}^+}),\\
 v\big|_{y_n=0}=\phi,
 \end{array}
 \right.
 \end{equation}
 with the estimate:
 for any $k\in\mathbb{Z}_+$, $\alpha\in\mathbb{Z}_+^{n-1}$, there exists a constant $C_{k,\alpha}>0$ such that
 \begin{equation}
 \big|D_{y_n}^k D_{y'}^\alpha r^{(N)}(y,\tau)\big|\le C_{k,\alpha}\tau^{-N-1+k},\,\,\,(y',y_n)\in \mathbb{R}_{y'}^{n-1}\times\overline{\mathbb{R}_{y_n}^+}.
 \end{equation}
 
 Now by taking $N$ large enough, we denote by $v_N$ the approximate solution of \eqref{equation of v} near $\partial\Omega$ obtained by the following procedures. That is cutting off each of this local approximate solution $v^{(N)}$ attached to the local coordinates neighbourhood of $\partial\Omega$ and further pulling it back to the original coordinates, and then patching these local approximate solutions by a partition of unity. Then since each $v^{(N)}$ decays exponentially as $\tau\rightarrow\infty$ away from $y_n=0$, we have
 \begin{equation}
 \left\{
 \begin{array}{ll}
 \tau^2 v_N-\nabla\cdot(\gamma\nabla v_N)=r_N\,\,&\text{in}\,\,\Omega,\\
 v_N=\phi\,\,&\text{on}\,\,\partial\Omega
 \end{array}
 \right.
 \end{equation}
 with the estimate:
 for any $k\in\mathbb{Z}_+$, $\alpha\in\mathbb{Z}_+^{n-1}$, there exists a constant $C_{k,\alpha}>0$ such that
 \begin{equation}
 \big|D_{y_3}^k D_{y'}^\alpha r_N\big|\le C_{k,\alpha}\tau^{-N-1+k}\,\, \text{on}\,\,\overline\Omega.
 \end{equation}
 Then, the representation $v=v_N+w_N$ of $v$  with $w_N$ solving
 \begin{equation}\label{well-posedness for w}
 \left\{
 \begin{array}{ll}
 \tau^2 w_N-\nabla\cdot(\gamma\nabla w_N)=-r_N\,\,&\text{in}\,\,\Omega,\\
 w_N=0\,\,&\text{on}\,\,\partial\Omega
 \end{array}
 \right.
 \end{equation}
 gives a good asymptotic behaviour $v$ as
 $\tau\rightarrow\infty$ by the well-posedness of \eqref{well-posedness for w}.
Therefore we have achieved our task. 
 \begin{remark}${}$
 \par
 The argument deriving the asymptotic property of $v(\tau)$ using the theory of pseudodifferential operators is a constructive argument giving the dominant part of $v(\tau)$. It can be applied to  the case $|\tau|\rightarrow\infty$ in a subdomain $S^i$ of a sector containing $\{\tau\in\mathbb{R}: \tau\ge 1\}$ and it can also show that $v(\tau)$ is analytic with respect to $\tau\in S^i$. The precise definition of $S^i$ is given in Subsection \ref{subsection C}. Important necessary changes are as follows. In the definition of $S(m)$, $\tau$ in $a=a(y,\eta',\tau)$ must be considered in $S$ and $a$ is analytic with respect to $\tau\in S^i$. All other necessary changes can be easily figured out from the contexts.
 \end{remark}
 
 \section{}\label{subsection C}
 \setcounter{equation}{0}
 \renewcommand{\theequation}{C.\arabic{equation}}
 We will give here the analyticity of $\widehat{u_1^f}$, $\widehat{u_1^g}$ with respect to $\tau$ in an unbounded subdomain of a sector containing $\{\tau\in\mathbb{R}: \tau\ge 1\}$ and a useful lemma to derive the integral identity \eqref{integral identity}. To begin with let us recall $S$ and $S^i$. They were defined as $S:=\{\tau=\tau_R+\sqrt{-1}\,\tau_I\in\mathbb{C}: \tau_R,\,\tau_I\in\mathbb{R},\,
 \tau_R\ge1,\,|\tau_I|<\kappa|\tau_R|\}$ with a fixed small $\kappa>0$ and $S^i=\text{interior of $S$}$. Then we first show the analyticity of $\widehat{u_1^f}$, $\widehat{u_1^g}$ which of course enough to show
 it for the solution of \eqref{equation of v}. Let $\text{Tr}^{-1}: H^{m+1/2}(\partial\Omega)\rightarrow H^{m+1}(\Omega)$ with $m\in\mathbb{Z}_{+}$ be the inverse trace operator. Here note that $m\in\mathbb{Z}_+$ can be taken arbitrarily. Represent the solution $v=v(\tau)\in H^1(\Omega)$ of \eqref{equation of v} as $v=\widetilde v+V$ with $V=\text{Tr}^{-1}\phi\in H^1(\Omega)$. Then $\widetilde v$ has to satisfy
 \begin{equation}\label{equation for tilde v}
 \left\{
 \begin{array}{ll}
 (\tau^2-L_\gamma)\widetilde v=-(\tau^2-L_\gamma)V\,\,&\text{in}\,\,\Omega,\\
 \widetilde v=0\,\,&\text{on}\,\,\partial\Omega,
 \end{array}
 \right.
 \end{equation}
 where $L_\gamma\bullet=\nabla\cdot(\gamma\nabla\bullet)$. 
 
 Now let $G_\gamma$ be the Green function of $-L_\gamma$ with homogeneous Dirichlet boundary condition at $\partial\Omega$, which is an isomorphism from $H^{-1}(\Omega):=H_0^1(\Omega)^*$
 to $H_0^1(\Omega)$. Further it is a positive and compact operator on $L^2(\Omega)$. Hence there exist positive constants $\ell_0,\,m_0$ such that $\ell_0 I\le G_\gamma\le m_0 I$ with the identity operator $I$ on $L^2(\Omega)$. Then by $\tau_R^2-\tau_I^2>(1-\kappa^2)\tau_R^2$
 and $2|\tau_R\tau_I|<2\kappa\tau_r^2$, hence the operator $H_\gamma(\tau):=(\tau_R^2-\tau_I^2)G_\gamma+I$ is a positive operator on $L^2(\Omega)$ with the estimate:
 \begin{equation}
 H_\gamma(\tau)\ge \{(1-\kappa^2)\tau_R^2\ell_0+1\}I
 \end{equation}
 if $\kappa$ is small enough. Then \eqref{equation for tilde v} reduces to
 \begin{equation}\label{integral equation for tilde v}
 \big(I+2\sqrt{-1}\tau_R\tau_I H_\gamma(\tau)^{-1}G_\gamma\big)\widetilde v=-H_\gamma(\tau)^{-1}(\tau^2-L_\gamma)V.
 \end{equation}
 Since the operator norm of the operator $2\sqrt{-1}\tau_R\tau_\eta H_\gamma(\tau)^{-1}G_\gamma$ on $L^2(\Omega)$ is estimated from above by
 $2(m_0\kappa\tau_R^2)/\{(1-\kappa^2)\ell_0\tau_R^2+1\}$ which is smaller than $1$ for any $\tau\in S^i$ by taking $\kappa$ small enough. Hence we have a unique solution $\widetilde v(\tau)\in L^2(\Omega)$ of \eqref{equation for tilde v} for any $\tau\in S^i$ by the Neumann series of 
 $\big(I+2\sqrt{-1}\tau_R\tau_I H_\gamma(\tau)^{-1}G_\gamma\big)^{-1}$. From
 \eqref{integral equation for tilde v} we immediately have $\widetilde v(\tau)\in H_0^1(\Omega)$ for any $\tau\in S^i$ and its continuous dependency on $\tau\in S^i$. We remark here that we can also have the continuous dependency of $\widetilde v(\tau)$ on any inhomogeneous term in $H^{-1}(\Omega)$.
 
 To see the differentiability at $\tau\in S^i$, take $0\not=\sigma\in\mathbb{C}$ small enough and consider $r(\sigma):=(\widetilde v(\sigma+\tau)-\widetilde v(\tau))-\sigma\widetilde w(\tau)$, where $\widetilde w(\tau)\in H^1(\Omega)$ be the solution to
 \begin{equation}\label{equation for tilde w}
 \left\{
 \begin{array}{ll}
 (\tau^2-L_\gamma)\widetilde w(\tau)=-2\tau\widetilde v(\tau)-2\tau V\,\,&\text{in}\,\,\Omega,\\
 \widetilde w(\tau)=0\,\,&\text{on}\,\,\partial\Omega
 \end{array}
 \right.
 \end{equation}
 which can be obtained likewise $\widetilde v(\tau)$. Then by a direct computation,
 we have
 $$
 \left\{
 \begin{array}{ll}
 (\tau-L_\gamma)r(\sigma)=-\sigma^2(\widetilde v(\sigma+\tau)+V)-2\sigma\tau(\widetilde v(\sigma+\tau)-\widetilde v(\tau))\,\,&\text{in}\,\,\Omega,\\
 r(\sigma)=0\,\,&\text{on}\,\,\partial\Omega.
 \end{array}
 \right.
 $$
 Hence by the continuous dependency of $\widetilde v(\tau)$ on $\tau$ and also on
 any inhomogeneous term of its equation, we have $\sigma^{-1}r(\sigma)\rightarrow 0$ as $\sigma\rightarrow0$. Hence $\widetilde v(\tau)$ is differentiable at $\tau\in S^i$ and its derivative is $\widetilde w(\tau)$.
 
 Next we will formulate and prove a useful lemma for deriving \eqref{integral identity}.
 
 \begin{lemma}\label{Useful lemma}
 	Let $p(\tau),\,q(\tau)$ be analytic functions in $S^i$. We assume that these functions satisfy the following two conditions;
 	\begin{itemize}
 		\item [{\rm (i)}] $f(\tau)+g(\tau)=0$ in $S^i$.
 		\item [{\rm (ii)}] $g(\tau)=o(1)f(\tau)$ as $|\tau|\rightarrow\infty$.
 	\end{itemize}
 	Then $f(\tau)=0$ in $S^i$.
 \end{lemma}
 \begin{proof} Suppose $f(\tau)\not\equiv0$ in $S^i$. Let $R:=\{\tau=\tau_R+\sqrt{-1}\tau_I\in\mathbb{C}: \tau_R,\,\tau_I\in\mathbb{R}, \tau_R\ge2,\,|\tau_I|\le |\tau_R|/2 \}$. For $n\in\mathbb{N}$,
 	denote $D_n=R\cap\{\tau\in\mathbb{C}: n+2\le|\tau|\le n+3\}$, and $Z(h;\mathcal{D}):=\{\tau\in \mathcal{D}: h(\tau)=0\}$ for any function $h=h(\tau)$ defined on a set $\mathcal{D}$. Then, since $f(\tau)$ is analytic in $S^i$, $\sharp(Z(f;D_n))<\infty$ for each $n\in\mathbb{N}$, where $\sharp(E)$ for a set $E$ denotes the number of elements in
 	$E$. Hence we have a sequence $\tau_n,\,n\in\mathbb{N}$ such that each $\tau_n\in D_n^i$, $f(\tau_n)\not=0$. Then $\tau_n\rightarrow\infty$ as $n\rightarrow\infty$, and hence we have $|g(\tau_n)|<|f(\tau_n)|$ for large enough $n$ by the assumptions on $f(\tau),\,g(\tau)$.
 	Fix such an $n\in\mathbb{N}$ and consider a small simply connected domain $\mathcal{E}\subset D_n^i$ such that $\tau_n\in\partial\mathcal{E}$ and
 	$|g(\tau)|<|f(\tau)|,\,\tau\in\partial\mathcal{E}$, where $\partial\mathcal{E}$ denotes the boundary of $\mathcal{E}$. Then by the Rouche theorem, we have
 	$\sharp(Z(f+g;\mathcal{E}))=\sharp(Z(f;\mathcal{E}))$. But by the assumption (i), we have
 	$\sharp(Z(f+g;\mathcal{E}))=\infty$ which implies that $f\equiv0$ in $\mathcal{E}$
 	and hence $f\equiv0$ in $S^i$ by the analyticity of $f$ in $S^i$. This contradict to the assumption $f(\tau)\not\equiv0$ in $S^i$ and hence we have the conclusion of the lemma.
 \end{proof}
 
 \section{}
 \setcounter{equation}{0}
 \renewcommand{\theequation}{D.\arabic{equation}}
 We give here the estimates of $\widehat I_j,\,j=1,2,3,4$. Since the arguments are almost the same for estimating all of them, we only give it for $\widehat I_2$. Recall the form of $I_2$. It was given as
 \begin{equation}\label{form of I_2}
 \begin{array}{ll}
 I_{2}=-4\sum_{j=1}^{n}B^{-1/2}\int\limits_{0}^{s}\sin\left\{(s-\sigma)B^{1/2}\right\}\\
 \qquad\qquad\qquad\quad\partial_{j}\Bigg[\sum_{k,l=1}^{n}c_{kl}^{j}(x)\Big\{(\partial_{k}u_{1}^{f})'(\sigma,x)\partial_{l}u_{1}^{g}(s-\sigma,x)\\
 \qquad\qquad\qquad\qquad\qquad\qquad\quad+\partial_{k}u_{1}^{g}(s-\sigma,x)(\partial_{l}u_{1}^{f})'(\sigma,x)\Big\}\Bigg]d\sigma\Bigg) dt.
 \end{array}
 \end{equation}
 
 Since the argument is the same for the other term, we will only analyze the Laplace transform $\widehat{I_2'}(\tau)$ of
 $$I_2':=-4\sum_{j=1}^{n}B^{-1/2}\int\limits_{0}^{s}\sin\left\{(s-\sigma)B^{1/2}\right\}\partial_j\Bigg[\sum_{k,l=1}^{n}c_{kl}^j(x)\partial_{k}u_{1}^{g}(s-\sigma,x)(\partial_{l}u_{1}^{f})'(\sigma,x)\Bigg]d\sigma\Bigg) dt.
 $$
 By \eqref{derivation of dominant part of I_2}, we have $\widehat{I_2'}(\tau)=\widehat{\mathcal{I}_2'}(\tau)+\widehat{\mathcal{R}_2'}(\tau)$,
 $$
 \widehat{\mathcal{I}_2'}(\tau)=-4\sum_{j=1}^{n}\partial_{j}\left[c_{kl}^{j}(x)\left(\widehat{F_{k}}\,\,\widehat{(\partial_{l}u_{1}^{f})'}\right)(\tau,x)\right]
 $$
 with $F_{k}(s,x)=s\partial_{k}u_{1}^{g}(s,x)$
 and $\widehat{\mathcal{R}_2'}(\tau)$ is the Laplace transform of $R_2'$ given by
 \begin{equation}
 R_2';=\int_0^s\Big[(s-\sigma)^2\int_0^\infty \lambda^{1/2}\lb\int_0^1 (1-\theta)\sin\left\{(s-\sigma)\lambda^{1/2}\right\}\,d\theta\rb dE(\lambda) G_2'(s,\sigma,x)\Big]\,d\sigma, 
 \end{equation}
 where
 $$
 G_2'(s,\sigma)=G_2'(s,\sigma,x):=4\sum_{j=1}^n\partial_j\left[\sum_{k,l=1}^{n}c_{kl}^j(x)\partial_{k}u_{1}^{g}(s-\sigma,x)(\partial_{l}u_{1}^{f})'(\sigma,x)\right].
 $$
 We will show that $\widehat{\mathcal{I}_2'}(\tau)$ is the dominant part of $\widehat{I_2'}(\tau)$. That is $\widehat{\mathcal{R}_2'}(\tau)$ is relatively smaller than this
 dominant part by $o(1)$ as $\tau\rightarrow\infty$.
 We begin by first estimating $H_2'(\sigma,s-\sigma)$ defined by
 $$
 H_2'(\sigma,s-\sigma)=\int_0^\infty \lambda^{1/2}\lb \int_0^1 (1-\theta)\sin\left\{(s-\sigma)\lambda^{1/2}\right\}\,d\theta\rb dE(\lambda) G_2'(s,\sigma,x).
 $$
 Then since each $c_{kl}^j\in C_0^\infty(\Omega)$, we have
 \begin{equation}\label{spectral estimate}
 \Vert H_2'(\sigma,s-\sigma)\Vert_{L^2(\Omega)}^2\le\int_0^\infty \D\lambda\, \lVert E(\lambda)G_2'(s,\sigma)\rVert_{L^2(\Omega)}^2=\lVert G_2'(s,\sigma)\rVert_{H^1(\Omega)}^2.
 \end{equation}
Here note that $\lVert G_2'(s,\sigma)\rVert_{H^1(\Omega)}=O(|\tau|^{5-2\mu)}$ as $S^i\ni\tau\rightarrow\infty$, where $S^i$ is the interior of an unbounded domain $S$ of a sector containing $\{\tau\in\mathbb{R}: \tau\ge1\}$ defined in Subsection \ref{subsection C}. This is coming from the asymptotic properties of $\widehat{u_1^f}(\tau),\,\widehat{u_1^g}(\tau)$ as
 $S^1\ni\tau\rightarrow\infty$ which we gave in Appendix B, and $\mu$ describes the decay of $\widehat\chi(\tau)=\tau^{-\mu}(1+O(\tau^{-1})$. 
 
 Now observe that 
 $$
 J(\tau):=\int_0^\infty e^{-\tau s}\left(\int_0^s (s-\sigma)^2 H_2'(\sigma,s-\sigma)\,d\sigma\right)\,ds=\int\limits_0^\infty e^{-\tau\sigma}\left(\int_0^\infty e^{-\tau s}s^2 H_2'(\sigma,s)\,ds\right)\,d\sigma.
 $$
 Here $H_2'(\sigma,s),\,s,\sigma\in [0,\infty)$ is an $L^2(\Omega)$ valued bounded measurable function and note that $\tau\in S^i$. Since
 there exists some $0<\delta<1$ such that  $s^2\,e^{-\tau_R\, s}\le \tau_R^{-2} e^{-\delta\tau_R\, s}$ for any real part $\tau_R$ of $\tau\in S^i$ and $|\tau|,\,\tau_R$ are equivalent for $\tau\in S^i$, we have by the Riemann-Lebesgue theorem,
 $$
 \Vert J(\tau)\Vert_{L^2(\Omega)}=o(|\tau|^{3-2\mu}),\,\,S^i\ni\tau\rightarrow\infty.
 $$
 On the other hand, it is easy to see that
 $$
 \Vert\widehat{\mathcal{I}_2'}(\tau)\Vert_{L^2(\Omega)}=O(|\tau|^{3-2\mu}),\,\,S^i\ni\tau\rightarrow
 \infty.
 $$
 Therefore $\widehat{\mathcal{I}_2'}(\tau)$ is the dominant.

\end{document}